\documentclass[11pt,fleqn]{amsart} 

\usepackage[usenames,dvipsnames,condensed]{xcolor}

\usepackage{amsthm}
\usepackage{amsfonts}
\usepackage[english]{babel}
\usepackage[usenames]{xcolor}
\usepackage{graphicx}
\usepackage{soul}
\usepackage{stfloats}
\usepackage{morefloats}
\usepackage{cite}
\usepackage{lscape}
\usepackage{epstopdf}
\usepackage{braket}
\usepackage[lite]{amsrefs}
\usepackage{mathbbol}

\usepackage{amsmath,amstext,amsopn,amsfonts,eucal,amssymb}
\usepackage{graphicx,wrapfig,url}

\newcommand\Z{{\mathbb Z}}

\newtheorem{theorem}{Theorem}[section]

\newtheorem{lemma}[theorem]{Lemma}
\newtheorem{proposition}[theorem]{Proposition}
\newtheorem{definition}[theorem]{Definition}

\newtheorem{example}[theorem]{Example}

\newtheorem{remark}[theorem]{Remark}

\begin{document}

\title{Braided Frobenius Algebras from Certain Hopf Algebras}

\author{Masahico Saito} 
\address{Department of Mathematics, 
	University of South Florida, Tampa, FL 33620, U.S.A.} 
\email{saito@usf.edu} 

\author{Emanuele Zappala} 
	\address{Institute of Mathematics and Statistics, University of Tartu\\
	Narva mnt 18, 51009 Tartu, Estonia} 
\email{emanuele.amedeo.zappala@ut.ee \\ zae@usf.edu}

\begin{abstract}
A braided Frobenius algebra is a Frobenius algebra with braiding that commutes with 
the  operations, that are related to diagrams of compact surfaces with boundary expressed as ribbon graphs. A heap is a ternary operation exemplified by a group with the operation
$(x,y,z) \mapsto xy^{-1}z$, that is ternary self-distributive.
Hopf algebras  can be endowed with
 the algebra version of the heap operation. Using this, we construct 
braided Frobenius algebras from a class of certain Hopf algebras  that admit integrals and cointegrals.  For these Hopf algebras we show that the heap operation induces a braiding, by means of a Yang-Baxter operator on the tensor product, which satisfies the required compatibility conditions.
Diagrammatic methods are employed for proving commutativity between the braiding and Frobenius operations.
\end{abstract}

\maketitle

\date{\empty}

\tableofcontents

\section{Introduction}

Frobenius algebras have been studied 
in recent decades 
in relation to 2-dimensional topological quantum field theories (TQFTs)~\cite{Kock},
and to Khovanov homology~\cite{Khov} in knot theory, that is a categorification of the Jones polynomial~\cite{Jones}.
Braid groups have been extensively used in relation to generalizations of the Jones polynomial, and 
braided monoidal categories have been developed 
to further extend knot invariants  to ribbon graphs~\cite{RT}, that consist of disk vertices and ribbon edges. 
Spatial graphs with a move that corresponds to handle slides have been studied for handlebody-links \cite{Ishii}. Corresponding algebraic structures that have multiplication and braiding at the same time, with compatibility conditions, have also been studied~\cite{CIST,Lebed}. 
Compact surfaces with boundary can be represented by ribbon graphs, and their moves~\cite{Matsu} and their invariants~\cite{IMM} have been studied. 
For  algebraic objects having both Frobenius and braiding structures,
Frobenius objects in braided monoidal categories   was proposed  in \cite{Comeau},
and relations to a certain tangle category were discussed.

\begin{figure}[htb]
\begin{center}
\includegraphics[width=2.8in]{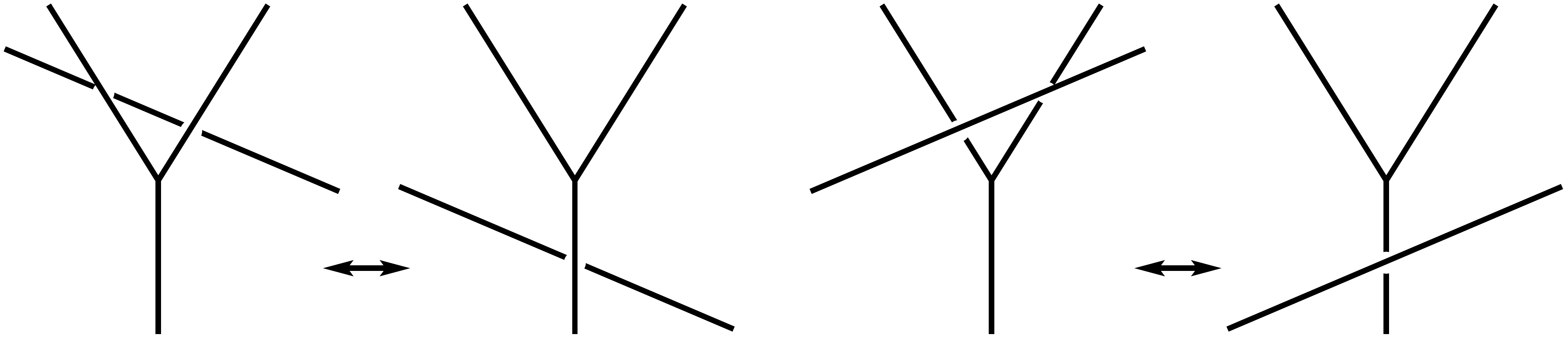}
\end{center}
\caption{Axioms of a braided Frobenius algebra}
\label{BF}
\end{figure}

Motivated from these developments, in this paper, we present  a construction of 
braided Frobenius algebras from certain Hopf algebras.
A  {\it   braided   Frobenius algebra}  
is a Frobenius object in the braided strict monoidal category of  
modules over 
unital rings
 (Definition~\ref{def:braidFrob}).
Specifically, a braided Frobenius algebra is a Frobenius algebra $X=(V, \mu, \eta, \Delta, \epsilon)$  (multiplication, unit, comltiplication, counit)  over a 
unital ring
 ${\mathbb k}$, which commute with the braiding, as explicitly formulated below. 
This commutation is represented by diagrams depicted in Figure~\ref{BF}, where the multiplication and braiding are represented by trivalent vertices and crossings, respectively, and these are part of moves for spatial graph diagrams. 

\begin{figure}[htb]
\begin{center}
\includegraphics[width=3in]{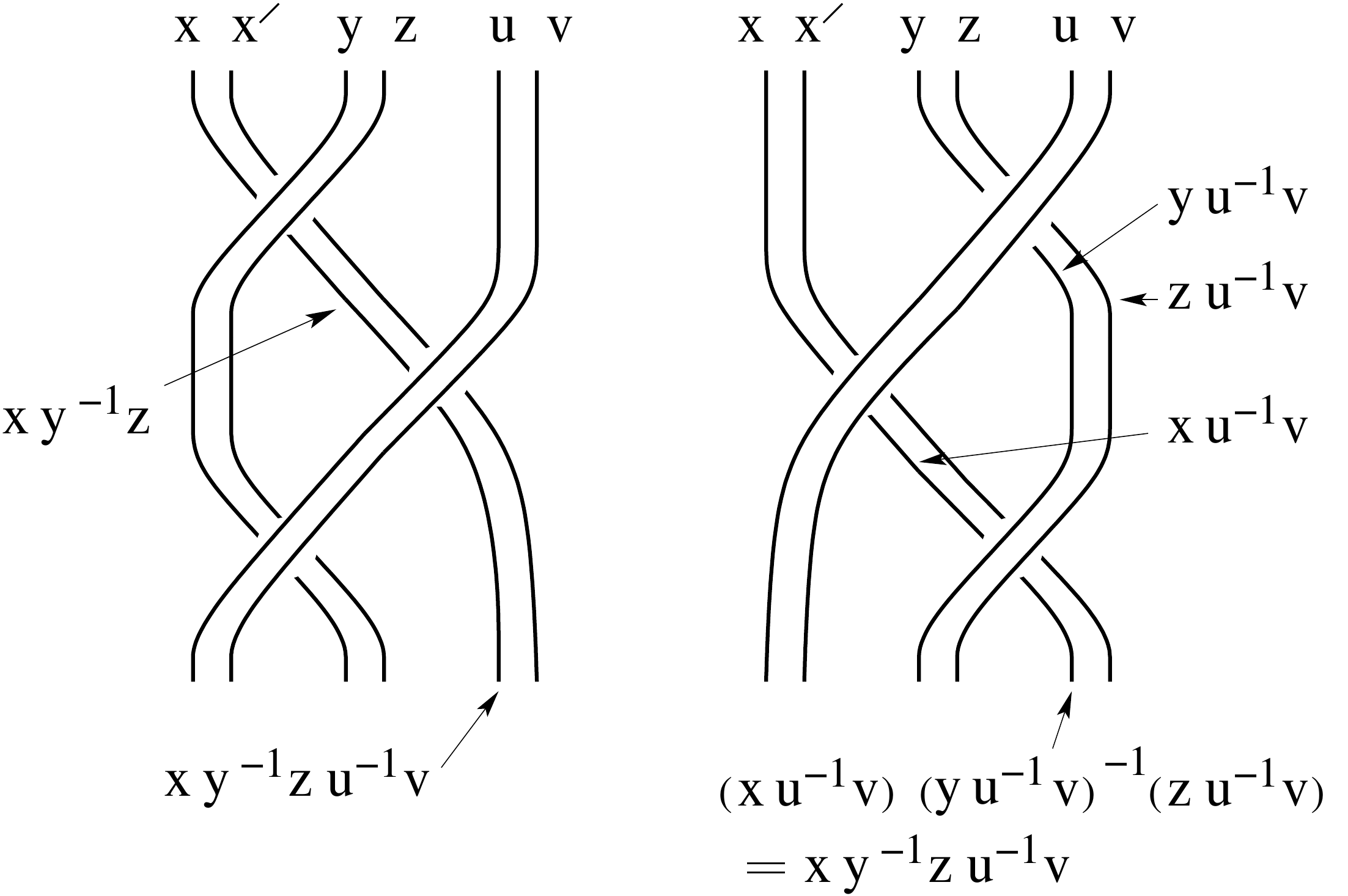}
\end{center}
\caption{Heap operation and braid relation}
\label{heaptypeIII}
\end{figure}

The idea of the construction is based on {\it heaps}. 
A heap is an abstraction of a group endowed with the ternary operation $a\times b\times c \mapsto T(a,b,c)= ab^{-1}c$. 
It is computed that this operation on a group satisfies the {\it ternary self-distributive law} (TSD)
$ T ( ( x, y,z ) ,  u,v)= T ( T(x,u,v), T(y,u,v) T(z,u,v ) )$ for all $x,y,z,u,v$.
Binary self-distributive operations have been studied in relation to the Yang-Baxter operators 
through tensor categories (e.g., \cite{CCES}).
 In \cite{ESZ} a diagrammatic interpretation of TSD was given in terms of framed links,
providing set-theoretic Yang-Baxter operators. 
The assignment of heap elements on arcs and the heap operations to crossings are
depicted in Figure~\ref{heaptypeIII}, together with the TSD property corresponding to
a braid relation (the type III Reidemeister move in knot theory). 
In \cite{ESZheap},  the constructions of TSD operations
from heaps were generalized to monoidal categories. Those in the category of finite dimensional Hopf algebras over a field are called {\it quantum heaps}. 
We use quantum heaps $X$ to construct 
a Frobenius algebra structure 
on $V=X \otimes X$ 
that commute 
with 
braiding induced from the TSD operations.
A key method of proofs is extensive  use of diagrams. 

The paper is organized as follows. In Section~\ref{sec:Pre} we review basic definitions and facts regarding heap structures, Hopf algebras, Frobenius algebras and Yang-Baxter operators. In Section~\ref{sec:TSD} we deal with ternary self-distributive (TSD) structures in coalgebras, 
and  construct a Yang-Baxter operator associated to a TSD structure arising from quantum heaps in Hopf algebras. 
In Section~\ref{sec:Braid} (co)pairings are constructed that commute with braidings.
These (co)pairing are used for (co)units for Frobenius structures.
In Section~\ref{sec:BraidFrob} we introduce the notion of braided Frobenius algebra and show that there is a class of these structures arising from quantum heaps where a Frobenius algebra is defined via Hopf algebra (co)integrals. Section~\ref{sec:Twist} discusses relations to compact surfeces with boundary embedded in 3-space, and 
issues of twists in braided Frobenius algebras.

\section{Preliminary}\label{sec:Pre}

In this section we review materials used in this paper. 

\subsection{Heaps}

We recall the definition and basic properties of heaps.
Given a set $X$ with a ternary operation $[ - ]$, the set of equalities 
$$ 
 [  [ x_1, x_2, x_3 ], x_4, x_5  ]=  [  x_1,[  x_4, x_3, x_2 ]  ,  x_5 ] = [  x_1,x_2, [  x_3, x_4, x_5 ]  ]   
$$ 
 is called para-associativity.
The 
equations $[x,x,y]=y$ and $[x, y, y ] = x$ are called the degeneracy conditions.
A {\it heap} is a non-empty set with a ternary operation satisfying 
		the para-associativity  
		and the degeneracy conditions~\cite{ESZheap}.
A typical example of a heap is a group $G$ where the ternary operation is given by $[x,y,z]=xy^{-1}z$,
which we call a {\it group heap}. 

Let $X$ be a set   with a ternary operation  $(x,y,z)\mapsto  T(x,y,z)$.
The condition 
$ T ( ( x, y,z ) ,  u,v)= T ( T(x,u,v), T(y,u,v) T(z,u,v ) )$ for all $x,y,z,u,v \in X$, is called {\it  ternary 
	 self-distributivity}, TSD for short. 
It is known and easily checked that the heap operation $(x,y,z)\mapsto  [x,y,z]=T(x,y,z)$ is ternary self-distributive. We focus on the TSD property of heaps.

\subsection{Hopf algebras}

A {\it Hopf algebra} $(X, \mu,  \eta,  \Delta, \epsilon, S)$ (a 
module over a unital ring 
$\mathbb k$,
multiplication, unit, comultiplication, counit, antipode,  respectively), is
defined as follows. 
First, a bialgebra 
$X$  
has a multiplication $\mu: X\otimes X\longrightarrow X$ with unit $\eta$ and a comultiplication $\Delta: X\longrightarrow X\otimes X$ with counit $\epsilon$ such that the compatibility condition $\Delta \circ\mu = (\mu\otimes \mu)\tau\circ (\Delta\otimes \Delta)$ holds. Then a Hopf algebra is a bialgebra endowed with a map $S: X\longrightarrow X$, called {\it antipode}, satisfying the equations $\mu \circ (\mathbb 1\otimes S)\circ \Delta = \eta \circ \epsilon = \mu\circ (S\otimes \mathbb 1)\circ\Delta$, called the {\it antipode condition}.

The diagrammatic representation of the algebraic operations appearing in a Hopf algebra is given in Figure~\ref{hopfop}. 
Diagrams are read from top to bottom. For example, the top two arcs 
of the trivalent vertex for $\mu$ (the leftmost diagram) represent $X \otimes X$, the vertex represents $\mu$, and the bottom arc represents $X$. 
In Figure~\ref{hopfaxiom} some of the defining axioms of a Hopf algebra are translated into diagrammatic equalities. Specifically, 
diagrams 
represent
(A) associativity of $\mu$, (B) unit condition, (C), compatibility between $\mu$ and $\Delta$, (D) the antipode condition. The coassociativity and  counit conditions are represented by diagrams that are vertical mirrors of (A) and (B), respectively. 

\begin{figure}[htb]
\begin{center}
\includegraphics[width=2.2in]{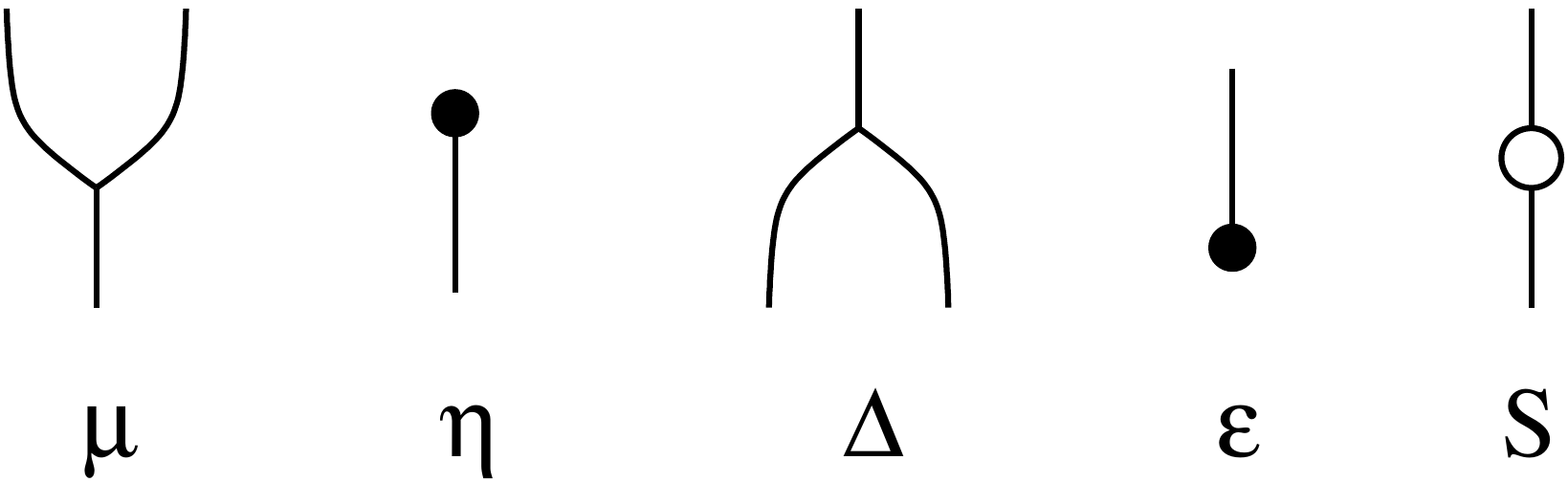}
\end{center}
\caption{Operations of Hopf algebras}
\label{hopfop}
\end{figure}

\begin{figure}[htb]
\begin{center}
\includegraphics[width=4.6in]{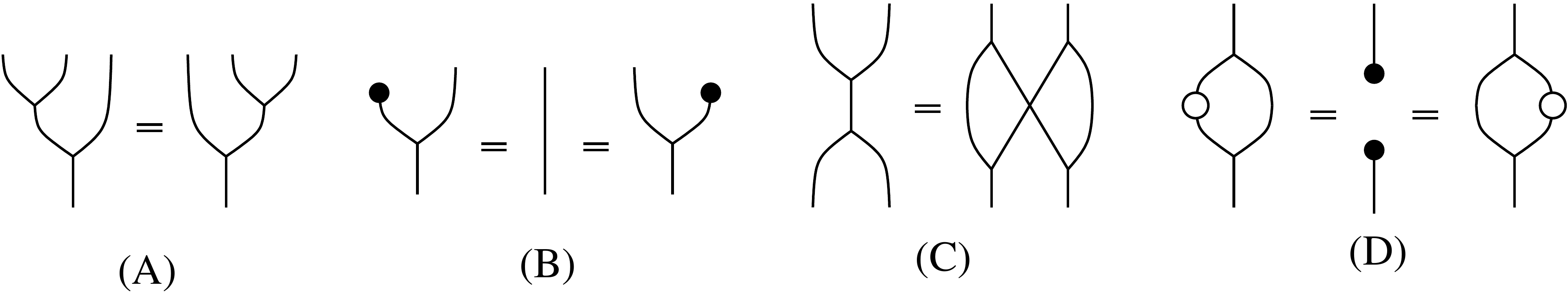}
\end{center}
\caption{Axioms of Hopf algebras}
\label{hopfaxiom}
\end{figure}

Any Hopf algebra satisfies the equality $S \mu  \tau = \mu (S \otimes S)$, where $\tau$ denotes the transposition $\tau(x \otimes y) = y\otimes x$ for simple tensors. 
This equality is depicted in Figure~\ref{mutwist}.
A Hopf algebra is called {\it involutory} if $S^2={\mathbb 1}$, the identity. 
It is known, \cite{Kas} Theorem~III.3.4, that if a Hopf algebra is commutative or cocommutative it follows that it is also involutory. In what follows, we will not mention that our Hopf algebras are involutory when they are (co)commutative, and freely apply the fact that $S^2={\mathbb 1}$ without further mention. 

For the comultiplication, we use Sweedler's notation $\Delta(x)=x^{(1)}\otimes x^{(2)}$ supressing the summation. Further, we use
$( \Delta \otimes {\mathbb 1} ) \Delta (x) = ( x^{(11)} \otimes x^{(12)} ) \otimes x^{(2)}$ and
$( {\mathbb 1}  \otimes \Delta ) \Delta (x) =  x^{(1)} \otimes ( x^{(21)} \otimes x^{(22)} ) $,
both of which are also written as 
$ x^{(1)} \otimes x^{(2)}  \otimes x^{(3)}$ from the coassociativity. 

\begin{figure}[htb]
\begin{center}
\includegraphics[width=1in]{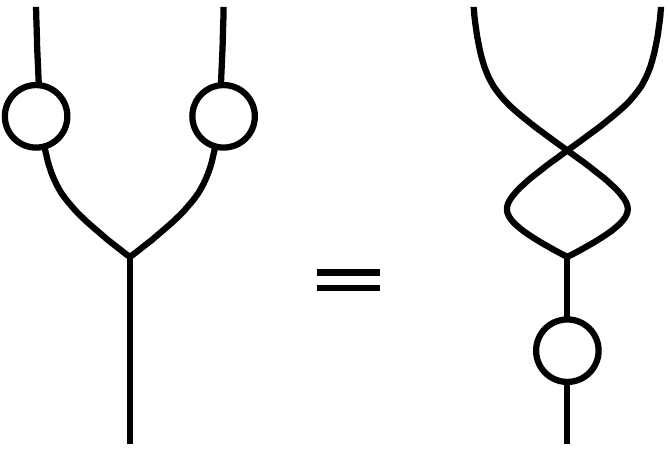}
\end{center}
\caption{Twisting $\mu$ with antipodes}
\label{mutwist}
\end{figure}

                                                                                                                                                                                                                                                                                                                                                                                                                                                                                                                                A  left integral of $X$ is an element $\lambda\in X$ such that $x\lambda = \epsilon (x) \lambda$ for all $x\in X$. A right integral, a (two-sided) integral, cointegrals are defined similarly.      
Diagrams for integral conditions are depicted in Figure~\ref{integ}. 
The diagram (A) represents an integral, (B) represents the defining equation of a left integral, 
and similar for cointegrals in (C) and (D). 
The existence of integrals is a fundamental tool to endow a Hopf algebra with a Frobenius structure (defined below). It is known that the set of integrals of a free finite dimensional 
Hopf algebra over a PID 
admits a one dimensional space of integrals, 
  see \cite{LS}. More generally, a finitely generated projective Hopf algebra over a ring admits a left integral space of rank one \cite{Par}. Observe that when a Hopf algebra is (co)commutative, it follows that a left (co)integral is also a right (co)integral.

\begin{figure}[htb]
\begin{center}
\includegraphics[width=3in]{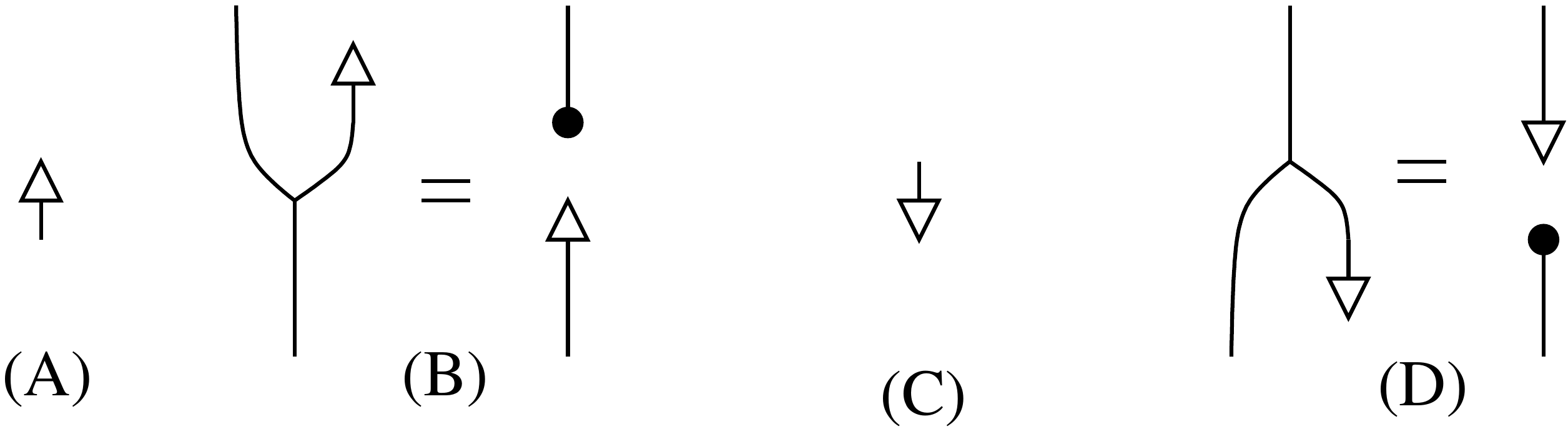}
\end{center}
\caption{Left (co)integral of Hopf algebras}
\label{integ}
\end{figure}

\subsection{Frobenius algebras}

We use the following definition: A {\it Frobenius algebra} $(V, \mu, \eta, \Delta, \epsilon)$ is an associative and coassociative coalgebra over a
unital 
 ring 
${\mathbb k}$ with multiplication $\mu$ and comultiplication $\Delta$, respectively, with unit 
$\eta: {\mathbb k} \rightarrow V$ and counit $\epsilon: V \rightarrow {\mathbb k}$
with the same conditions as Hopf algebras, such that $\mu$ and $\Delta$ satisfy the Frobenius compatibility condition:
$\ (\mu \otimes {\mathbb 1} ) (  {\mathbb 1}  \otimes \Delta) =\Delta \mu 
= ( {\mathbb 1} \otimes \mu) ( \Delta \otimes  {\mathbb 1} )$. 
Thid condition is depicted in Figure~\ref{frobcompati}.

\begin{figure}[htb]
\begin{center}
\includegraphics[width=1.5in]{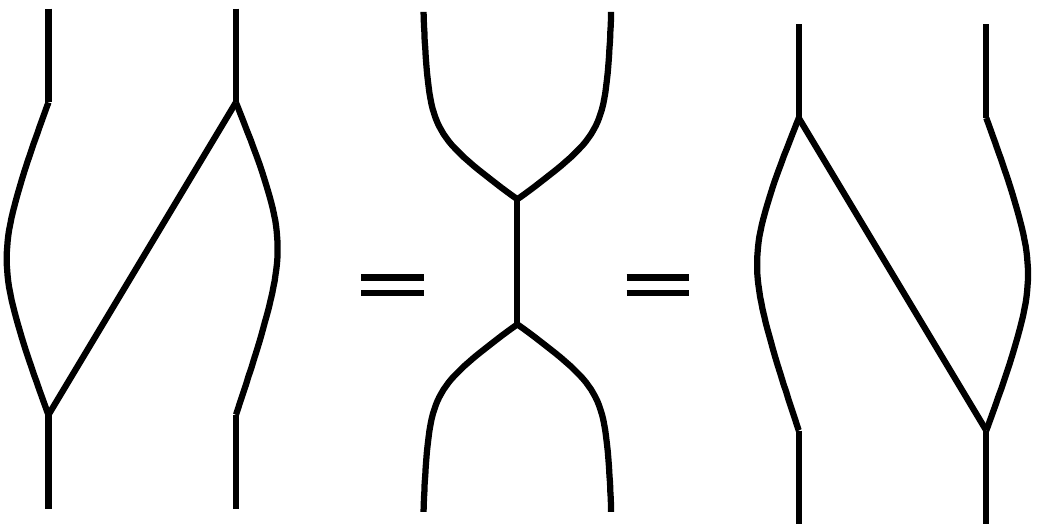}
\end{center}
\caption{ Frobenius  compatibility condition}
\label{frobcompati}
\end{figure}

\subsection{The Yang-Baxter operator}

Let $X$ be a 
module over a ring 
and let $R: X\otimes X\longrightarrow X\otimes X$ be an operator (i.e. a linear map). The Yang-Baxter equation, YBE for short, for $R$ is the functional equation
$$
(R\otimes \mathbb 1)\circ (\mathbb 1\otimes R) \circ (R\otimes \mathbb 1) = (\mathbb 1\otimes R) \circ (R\otimes \mathbb 1) \circ (\mathbb 1\otimes R)
$$
where LHS and RHS are both endomorpshism of $X\otimes X\otimes X$. The YBE is well known to be
represented by 
the type III Reidemeister move in knot theory, and has been widely studied in low-dimensional topology because it produces invariants of knots. If the operator $R$ satisfies the YBE, then it is said to be a {\it pre Yang-Baxter operator}. If, in addition, $R$ is invertible then we say that $R$ is a {\it Yang-Baxter operator}, YB operator for short. 

\section{Ternary self-distributive operations in coalgebras and braidings}\label{sec:TSD}

In this section we provide a method of producing braidings from ternary self-distributive (TSD) operations.

\begin{definition}\cite{ESZ} \label{def:TSD}
{\rm 
A morphism $T: V^{\otimes 3} \rightarrow V$ for a coalgebra $V$ over a unital ring ${\mathbb k}$ 
is called {\it ternary self-distributive} (TSD for short) if 
it satisfies, when expressed in simple tensors, 
\begin{eqnarray*}
\lefteqn{
T(T(x \otimes y \otimes z ) \otimes u \otimes v ) } \\
& = & T(T ( x \otimes u^{(11)}  \otimes v^{(11)} ) \otimes
T ( x \otimes u^{(12)}  \otimes v^{(12)} ) \otimes
T ( x \otimes u^{(2)}  \otimes v^{(2)} ) ) .
\end{eqnarray*}
}
\end{definition}

\begin{lemma}\cite{ESZ} \label{lem:T} 
Let $(X, \mu, \Delta, \iota, \epsilon, S) $ be an involutory Hopf algebra.
Let $T(x \otimes y \otimes z) = x S(y) z = \mu ( \mu (x \otimes S(y) ) \otimes z) $ 
 expressed in simple tensors, 
where the concatenation denotes the 
multiplication. 
Then $T$ is TSD. 
\end{lemma}

This construction is represented by the diagrams in Figure~\ref{qheap}.

\begin{figure}[htb]
\begin{center}
\includegraphics[width=1.5in]{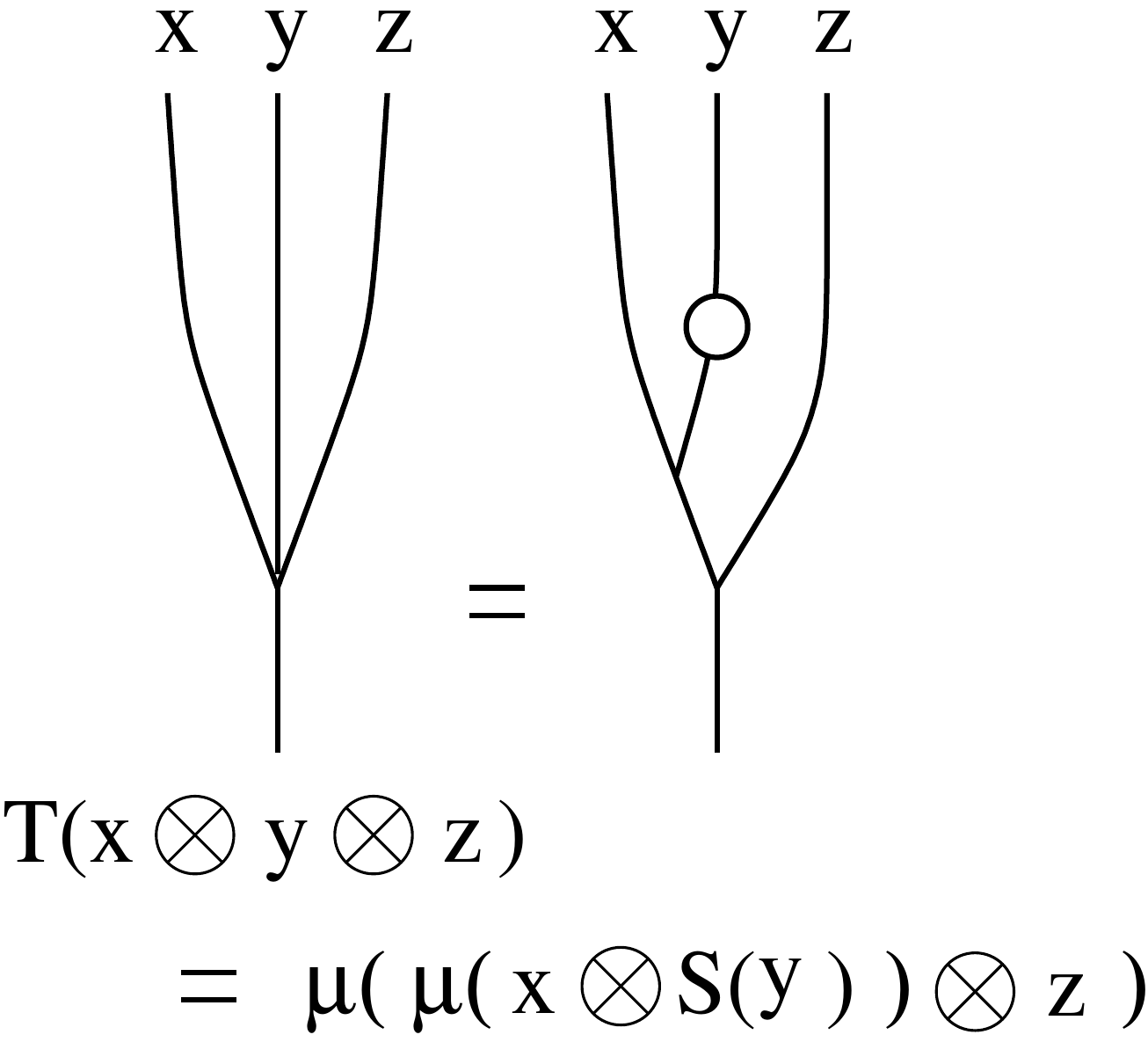}
\end{center}
\caption{Quantum heap operation as a TSD}
\label{qheap}
\end{figure}

\begin{definition}
{\rm
A TSD morphism 
   $T: V^{\otimes 3} \rightarrow V$ for a module $V$ over a unital ring ${\mathbb k}$ 
is called {\it invertible } if it satisfies 
$$
T(T(x\otimes y^{(2)}\otimes z^{(2)})\otimes z^{(1)}\otimes y^{(1)}) = \epsilon(y)\epsilon(z) \cdot x,
$$
for all $x, y, z \in V$. 
}
\end{definition}


\begin{lemma}\label{lem:qheapinv}
Let $(X, \mu, \Delta, \iota, \epsilon, S) $ be an involutory
 Hopf algebra,
and let $T(x \otimes y \otimes z) = x S(y) z $ be as defined in Lemma~\ref{lem:T}. 
If $(X, \Delta)$ is cocommutative, then $T$ is invertible.
\end{lemma}

\begin{proof}
One computes 
\begin{eqnarray*}
\lefteqn{ T(T(x\otimes y^{(2)}\otimes z^{(2)})\otimes z^{(1)}\otimes y^{(1)}) }\\
& & 
=  x S(  y^{(2)})  z^{(2)} S( z^{(1)} ) y^{(1)}
= x S(  y^{(2)})  S (  z^{(1)}  ) z^{(2)} y^{(1)}
= \epsilon(z) \cdot x S(  y^{(2)})  y^{(1)}
 = \epsilon(y)\epsilon(z) \cdot x
\end{eqnarray*}
as desired. 
\end{proof}

\begin{figure}[htb]
\begin{center}
\includegraphics[width=6in]{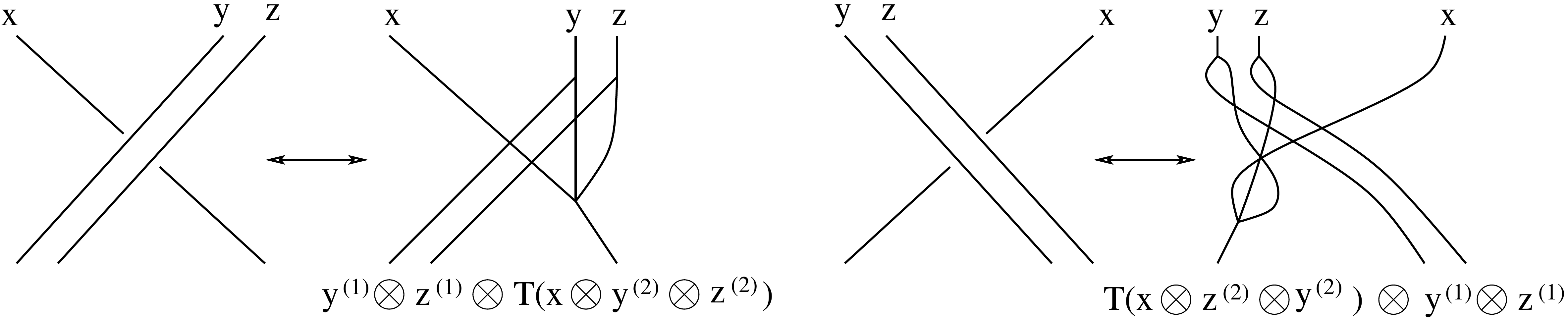}
\end{center}
\caption{Hopf algebra maps corresponding to crossings}
\label{crossingdetail}
\end{figure}

\begin{lemma}\label{lem:singletypeII}
	Let $(X, \Delta) $ be a cocommutative coalgebra  over a unital ring ${\mathbb k}$.
	Let $T: X^{\otimes 3} \rightarrow X$ be an invertible  TSD coalgebra morphism.
	Then the map $\beta_1: X^{\otimes 3} \rightarrow  X^{\otimes 3}$ defined for 
	simple tensors
	by 
	$\beta_1 (x \otimes y \otimes z) = y^{(1)} \otimes z^{(1)} \otimes T(x \otimes y^{(2)} \otimes z^{(2)}) $
	is invertible with the inverse 
	$\beta_1^{-1} ( y \otimes z \otimes x) = T(x \otimes z^{(2)} \otimes y^{(2)} ) \otimes y^{(1)} \otimes z^{(1)}$, so that $\beta_1 \beta_1^{-1}={\mathbb 1}$ and $ \beta_1^{-1} \beta_1 ={\mathbb 1}$.
	
\end{lemma}

\begin{proof}
	The proof is an application of the invertibility condition of $T$. On simple tensors we have
	\begin{eqnarray*}
		\beta_1^{-1}\beta_1(x\otimes y\otimes z) &=&\beta_1^{-1}( y^{(1)}\otimes z^{(1)}\otimes T(x\otimes y^{(2)}\otimes z^{(2)}))\\
		&=& T(T(x\otimes y^{(2)}\otimes z^{(2)})\otimes z^{(12)}\otimes y^{(12)})\otimes y^{(11)}\otimes z^{(11)}\\
			&=& T(T(x\otimes y^{(22)}\otimes z^{(22)})\otimes z^{(21)}\otimes y^{(21)})\otimes y^{(1)}\otimes z^{(1)}\\
		&=& \epsilon(y^{(2)})\epsilon (z^{(2)}) x\otimes y^{(1)}\otimes z^{(1)}\\
		&=& x\otimes y\otimes z,
	\end{eqnarray*}
	which shows that $\beta_1^{-1} \beta_1 = \mathbb 1$. Similar considerations imply that $\beta_1\beta_1^{-1} = \mathbb 1$ as well. 
\end{proof}

Diagrammatic representations of  morphisms $\beta_1$ and $\beta_1^{-1}$ in Lemma~\ref{lem:singletypeII} are depicted in the left and right of Figure~\ref{crossingdetail}, 
respectively. The first equality $\beta_1^{-1} \beta_1=  \mathbb 1$ in the lemma is represented by Figure~\ref{typeIIhalf}. 

\begin{figure}[htb]
\begin{center}
\includegraphics[width=1in]{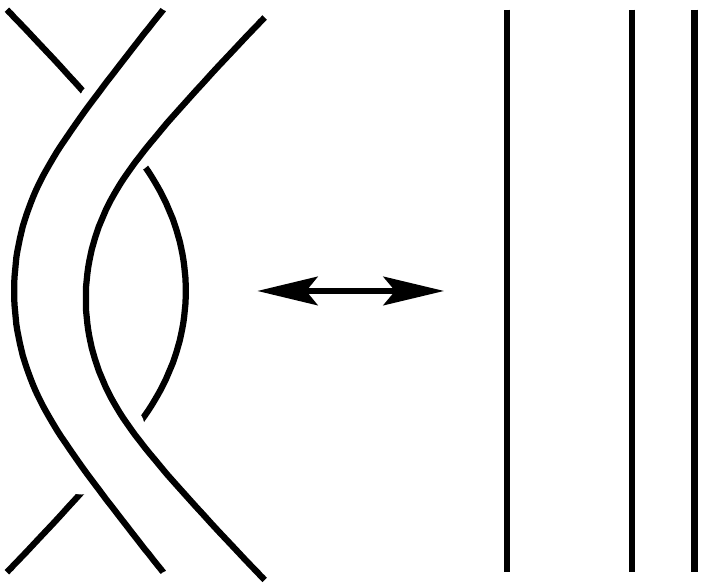}
\end{center}
\caption{The type II Reidemeister move with a single under-arc and double over-arcs.}
\label{typeIIhalf}
\end{figure} 
 
\begin{lemma}
 \label{lem:B} 
Let $(X, \Delta) $ be a cocommutative coalgebra  over a unital ring ${\mathbb k}$.
Let $T: X^{\otimes 3} \rightarrow X$ be an invertible  TSD coalgebra morphism.
Let $V=X \otimes X$ be endowed with the tensor coalgebra structure induced by $(X,\Delta)$.  
Then the map $\beta: V^{\otimes 2} \rightarrow  V^{\otimes 2}$ defined for 
simple tensors
 by 
$$\beta( ( x\otimes x')  \otimes ( y \otimes z)  ) := 
( y^{(1)} \otimes z^{(1)} ) \otimes 
T(x \otimes y^{(21)} \otimes z^{(21)} ) \otimes
T(x'  \otimes y^{(22)} \otimes z^{(22)} ) $$
satisfies the YBE. 
Furthermore, there is an inverse 
$$\beta^{-1} ( ( y \otimes z ) \otimes (x \otimes x') )
= (T ( x \otimes z^{(21)} \otimes y^{(21)} ) \otimes T ( x'  \otimes z^{(22)} \otimes y^{(22)} ) )
\otimes y^{(1)} \otimes z^{(1)} . $$
\end{lemma}
\begin{proof}
	We show that the YBE holds on simple tensors. Let $x,y,z,w,u,v\in X$, then the LHS of the YBE is computed as 
	\begin{eqnarray*}
	\lefteqn{(\beta\otimes \mathbb 1)(\mathbb 1\otimes \beta)(\beta\otimes \mathbb 1)(x\otimes y \otimes z\otimes w\otimes u\otimes v)}\\
	&=& u^{(1)}\otimes v^{(1)}\otimes T(z^{(1)}\otimes u^{(2)}\otimes v^{(2)})\otimes T(w^{(1)}\otimes u^{(3)}\otimes v^{(3)})\\
	&&\hspace{0.5cm} \otimes T(T(x\otimes z^{(2)}\otimes w^{(2)})\otimes u^{(4)}\otimes v^{(4)}) \\
	&&\hspace{0.5cm} \otimes T(T(y\otimes z^{(3)}\otimes w^{(3)})\otimes u^{(5)}\otimes v^{(5)}).
	\end{eqnarray*}
    The RHS computed on $x\otimes y\otimes z\otimes w\otimes u\otimes v$ gives 
    \begin{eqnarray*}
    	\lefteqn{(\mathbb 1\otimes \beta)\circ (\beta\otimes \mathbb 1)\circ (\mathbb 1\otimes \beta)(x\otimes y \otimes z\otimes w\otimes u\otimes v)}\\
    	&=& u^{(1)}\otimes v^{(1)}\otimes T(z^{(1)}\otimes u^{(4)}\otimes v^{(4)})\otimes T(w^{(1)}\otimes u^{(7)}\otimes v^{(7)}) \\
    	&& \otimes T(T(x\otimes u^{(2)}\otimes v^{(2)})\otimes T(z^{(2)}\otimes u^{(5)}\otimes v^{(5)})
    \otimes T(w^{(2)}\otimes u^{(8)}\otimes v^{(8)}))\\
    	&&\otimes T(T(y\otimes u^{(3)}\otimes v^{(3)})\otimes T(z^{(3)}\otimes u^{(6)}\otimes v^{(6)})
 \otimes T(w^{(3)}\otimes u^{(9)}\otimes v^{(9)})).
    \end{eqnarray*}
     Rearranging terms by means of the cocommutativity of $\Delta$, using the fact that $T$ is a coalgebra morphism and applying the TSD property of $T$, we see that the two terms coincide, showing that $\beta$ satisfies the YBE.
     
     To show that $\beta$ is invertible observe that, since $\Delta$ is cocommutative, one has $\beta = (\beta_1 \otimes \mathbb 1) \circ (\mathbb 1\otimes \beta_1)$. Similar considerations allow us to write $\beta^{-1}$ as composition of terms where $\beta_1^{-1}$ appears. An iteration of Lemma~\ref{lem:singletypeII} then shows that $\beta^{-1}$ is the inverse of $\beta$. 
\end{proof}

		Figure~\ref{crossingdetail} shows the diagrammatic interpretation of the braiding and its inverse in Lemma~\ref{lem:B} on a single edge of a ribbon. The full braiding, as well as its inverse, is obtained by repeating the procedure on both edges that delimit a ribbon.  

\begin{lemma}\label{lem:quantumbraid}
	Let $(X,\mu, \eta, \Delta, \epsilon, S)$ be an involutory Hopf algebra. Then the map $\beta : X^{\otimes 2}\rightarrow X^{\otimes 2}$ defined on simple tensors as 
	$$
	x\otimes y \otimes z\otimes w \mapsto z^{(1)}\otimes w^{(1)} \otimes xS(z^{(2)})w^{(2)} \otimes yS(z^{(3)})w^{(3)}
	$$
	is a Yang-Baxter operator. 
\end{lemma}
\begin{proof}
	The statement follows directly by applying Lemma~\ref{lem:B} to the quantum heap construction of Lemma~\ref{lem:T}.
	The invertibility follows from Lemma~\ref{lem:qheapinv}.
\end{proof}

\section{Braidings and pairings in quantum heaps}\label{sec:Braid}

In this section we introduce pairings and copairings that commute with braiding constructed 
in the preceding section. We construct such (co)pairing using integrals of Hopf algebras.

\begin{figure}[htb]
\begin{center}
\includegraphics[width=1.5in]{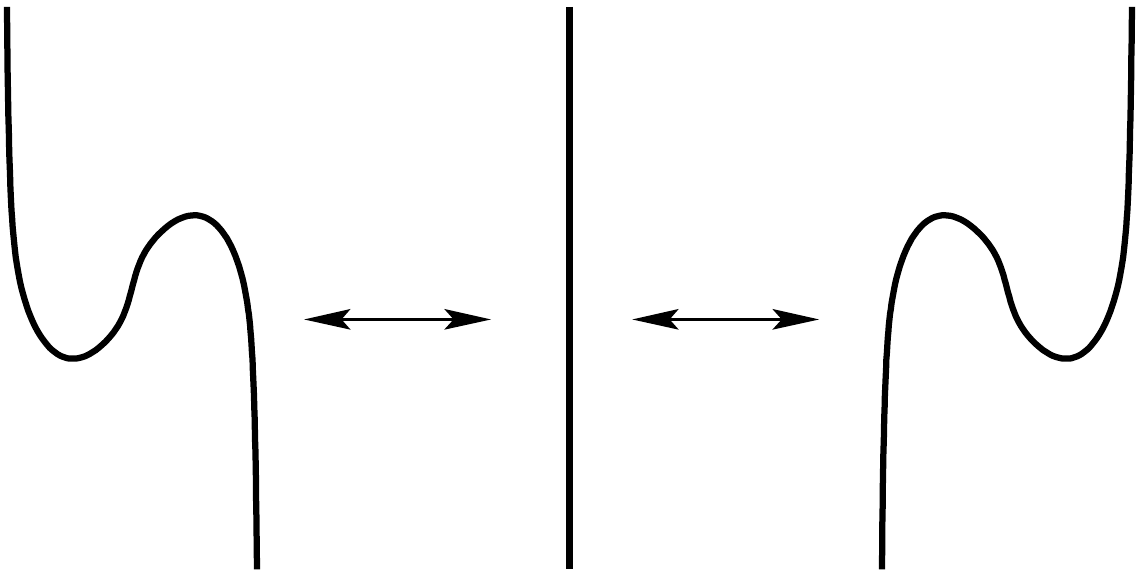}
\end{center}
\caption{The switchback  property}
\label{switchback}
\end{figure}

\begin{definition}
{\rm 
A pairing $\cup : V \otimes V \rightarrow \mathbb{k} $ and 
a copairing $\cap: {\mathbb k} \rightarrow V \otimes V$ in a  module $V$ over a unital ring $\mathbb{ k}$ are said to have (or satisfy) the {\it switchback property}
if they satisfy the equalities
$$(\cup \otimes {\mathbb 1}) ({\mathbb 1} \otimes \cap)={\mathbb 1} 
= ({\mathbb 1} \otimes \cup ) ( \cap  \otimes {\mathbb 1} ). $$
}
\end{definition}

The conditions imply  that $\cup$ is non-singular.

\begin{figure}[htb]
\begin{center}
\includegraphics[width=2.5in]{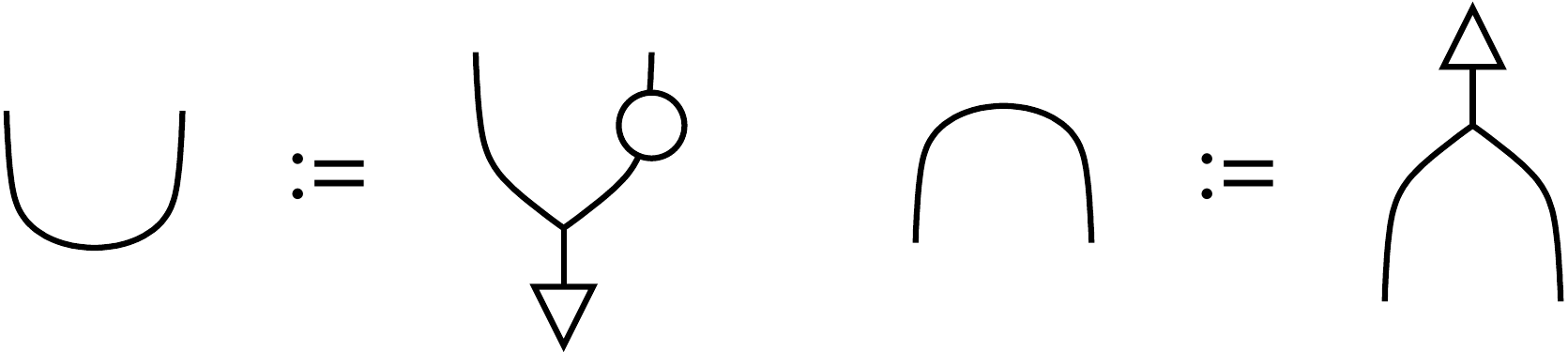}
\end{center}
\caption{Defining cup and cap by left (co)integrals}
\label{cupcap}
\end{figure}

\begin{definition}\label{def:cupcap}
{\rm 
Let  $(X,\mu, \eta, \Delta, \epsilon, S)$ be a
 finitely generated projective Hopf algebra over a (unital) ring $\mathbb k$.
Then $X$ has an integral and a cointegral \cite{Par}. Let us indicate them by $\lambda$ and $\gamma$, respectively. 
We define a cup on $X$ by 
 $\cup : = \lambda\mu(\mathbb 1\otimes S)$ and $\cap :=\Delta \gamma$, as depicted in Figure~\ref{cupcap}. 
 }
 \end{definition}

For a Hopf algebra $(X,\mu, \eta, \Delta, \epsilon, S)$,
the following module  $P(H^*)$  was considered in \cite{Par}.
Let $\chi: X^* \rightarrow X^* \otimes X$ be a right $X$-comodule structure on $X^*$ 
defined by the left $H^*$-module structure. 
Then $P(H^*)$ was defined by
$P(H^*)=\{ x^* \in X^* \mid \chi(x^*)=x^* \otimes 1 \}$.

\begin{lemma} 
	\label{lem:switchback}
	Let  $(X,\mu, \eta, \Delta, \epsilon, S)$ be a finitely generated projective Hopf algebra over a ring $\mathbb k$, 
	 such that $P(X^*)\cong {\mathbb k}$, 
	and $\cup$, $\cap$ be as  in Definition~\ref{def:cupcap}.
	Then $\cup$ and $\cap$ satisfy the switchback property. 
\end{lemma}

\begin{proof}
	In \cite{Par}, it is proved, under the assumptions, that there exists an integral $\lambda$ and 
	cointegral $\gamma$ such that $\cup'=\lambda \mu$ and $\cap' = (S \otimes {\mathbb 1}) \Delta \gamma$ satisfy the switchback property. It then follows that so do $\cup$ and $\cap$ in Definition~\ref{def:cupcap} as well. 
\end{proof}

Since we use this lemma extensively from here forward, we will assume that 
every Hopf algebra satisfies the assumption of this lemma. As pointed out in \cite{Par}, the condition that $P(H^*) \cong \mathbb k$ is automatically satisfied when ${\rm pic}(\mathbb k) = 0$. This is the case for instance when $\mathbb k$ is a PID or a local ring. In particular, one obtains the result of Larson and Sweedler in \cite{LS}, where the ground ring is taken to be a PID.



\begin{figure}[htb]
\begin{center}
\includegraphics[width=1.5in]{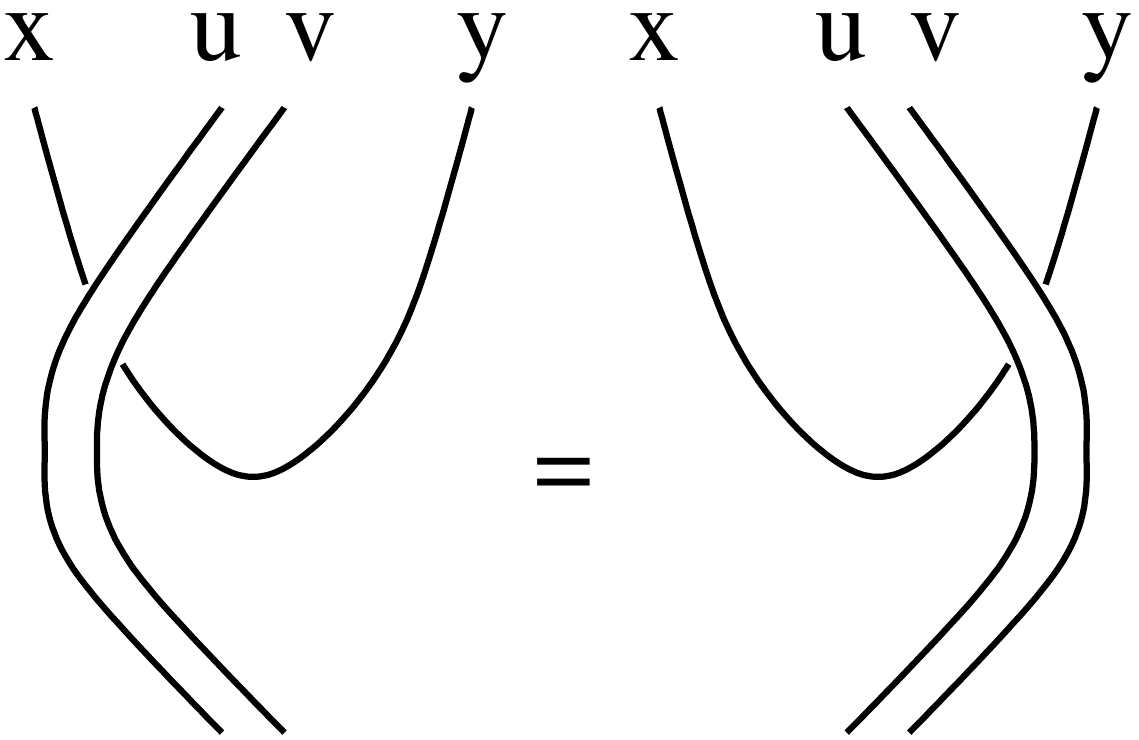}
\end{center}
\caption{The passcup property}
\label{passcup}
\end{figure}

\begin{definition}
{\rm  
Let $V$ be a coalgebra over a
unital ring $\mathbb{ k}$, with ternary morphism (of coalgebras)
$T: V^{\otimes 3} \rightarrow V$. 
 A pairing $\cup : V \otimes V \rightarrow \mathbb{k}$ is said to have (or satisfy) the {\it passcup  property with respect to $T$} if 
it satisfies 
$$ ( {\mathbb 1}^{\otimes 2} \otimes \cup  ) 
( u^{(1)} \otimes v^{(1)} \otimes T(x \otimes u^{(2)} \otimes  v^{(2)} ) \otimes y)
=( \cup \otimes {\mathbb 1}^{\otimes 2}  ) 
(x \otimes T(y \otimes v^{(2)} \otimes u^{(2)} ) \otimes u^{(1)} \otimes v^{(1)} ) $$
for all $x, y, u, v \in V$.
}
\end{definition}

The passcup property  is depicted in Figure~\ref{passcup}.

\begin{lemma} \label{lem:passcup}
Let  $(X,\mu, \eta, \Delta, \epsilon, S)$ be a cocommutative 
Hopf algebra. 
Then the pairing $\cup$ defined in Definition~\ref{def:cupcap}  satisfies the passcup property with respect to 
 the  TSD defined in Lemma~\ref{lem:T}. 
\end{lemma}

\begin{proof}
In order to prove the passcup property, we proceed as in Figure~\ref{passcupproof}. The first equality corresponds to rewriting one negative crossing using the definition of inverse of quantum heap operation $T$, the first arrow utilizes naturality of the switching map $X\otimes X \rightarrow X\otimes X$, the second arrow corresponds to the compatibility relation between the antipode $S$ and the comultiplication $\Delta$ of $X$, the third arrow is given by redrawing the diagram using naturality of switching map, the fourth arrow corresponds to the fact that $\lambda$ is both a right and left integral. Involutority is used at Step (3). 
This completes the proof of the passcup property.  
\end{proof}

\begin{figure}[htb]
\begin{center}
\includegraphics[width=6in]{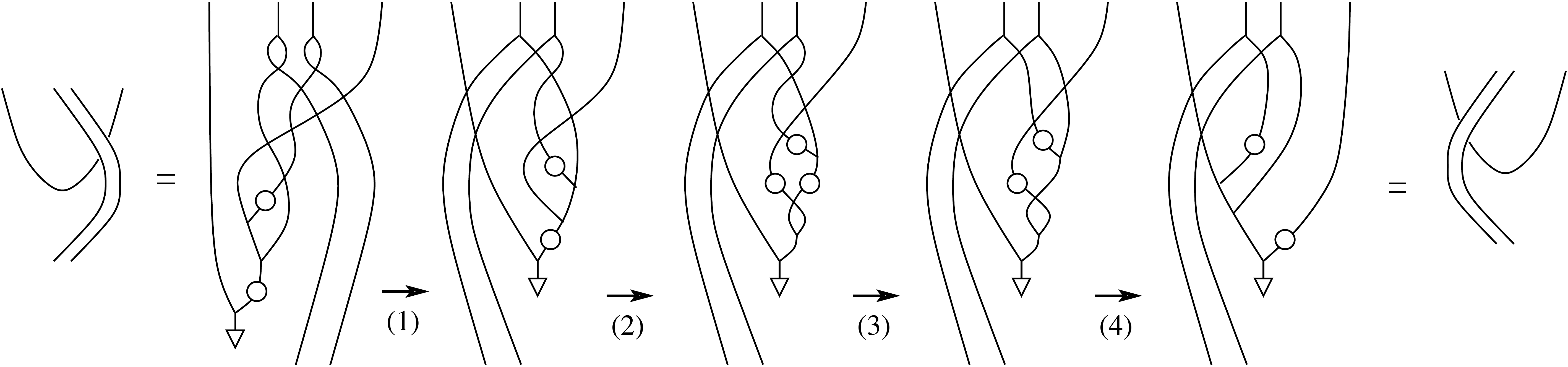}
\end{center}
\caption{Proof of the passcup property}
\label{passcupproof}
\end{figure}

\begin{lemma}\label{lem:cupcapbraid}
Let  $(X,\mu, \eta, \Delta, \epsilon, S)$ be a commutative and cocommutative Hopf algebra and set $V=X \otimes X$.
Then the  cup and cap defined in Definition~\ref{cupcap}
commute with the braiding defined in Lemma~\ref{lem:quantumbraid}.
Specifically, it holds that 
$({\mathbb 1} \otimes \cup) \beta = \cup \otimes {\mathbb 1} $  and 
$(\cup \otimes {\mathbb 1} ) \beta = {\mathbb 1} \otimes \cup$
as morphisms $V\otimes V \rightarrow V$, 
$({\mathbb 1} \otimes \cap) \beta = \cap \otimes {\mathbb 1} $ and 
$\beta (\cap \otimes {\mathbb 1} )=  {\mathbb 1} \otimes \cap$
as morphisms  $V \rightarrow V \otimes V$.
\end{lemma}

\begin{proof}
Diagrammatic sketch proofs are found in Figures~\ref{counitb2} through \ref{unitb1}.
In Figure~\ref{counitb2}, 
$({\mathbb 1} \otimes \cup) \beta = \cup \otimes {\mathbb 1} $ is proved by Lemmas~\ref{lem:passcup} and \ref{lem:singletypeII} successively.
Other equalities are proved as depicted, using Hopf algebra axioms and the definition of integrals.
In Figure~\ref{counitb1}, other than axioms, commutativity is used in the 4th equality, and cocommutativity is used in the 5th equality. While the diagrams in Figures~\ref{counitb1} and~\ref{unitb2} treat the single-stranded case, an iteration of the diagrammatic proof implies the case with two edges sliding.
 Cocommutativity is used in the 3rd and 5th equalities in Figure~\ref{unitb2},
and  3rd equality in  Figure~\ref{unitb1}. Observe that the equalities $S\eta=\eta$ and $\epsilon S=\epsilon$, which are consequences of $S$ being an anti-homomorphism, have been used. 
\end{proof}

\begin{figure}[htb]
\begin{center}
\includegraphics[width=3in]{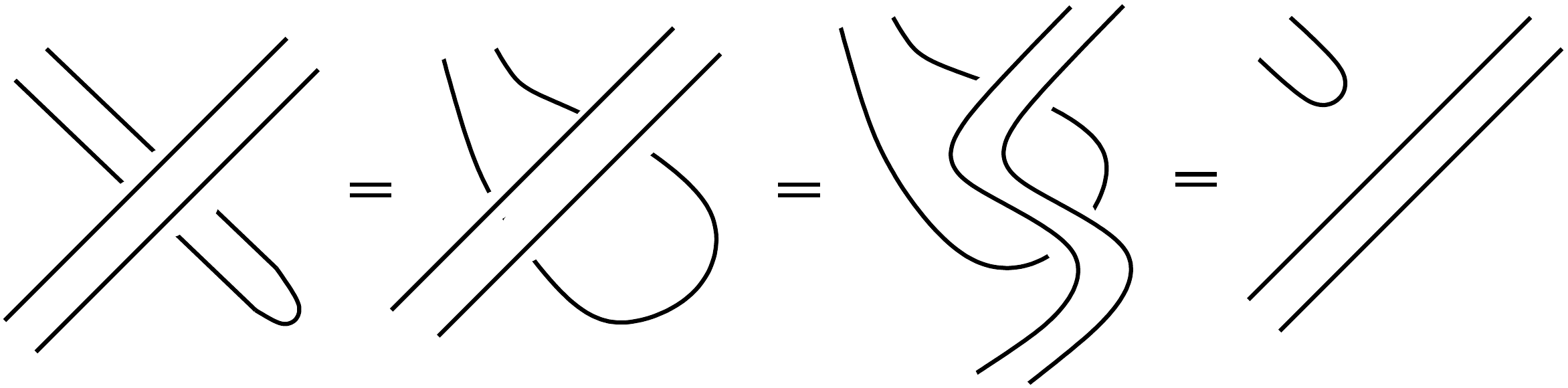}
\end{center}
\caption{Proof of $({\mathbb 1} \otimes \cup) \beta = \cup \otimes {\mathbb 1} $}
\label{counitb2}
\end{figure}

\begin{figure}[htb]
\begin{center}
\includegraphics[width=3.5in]{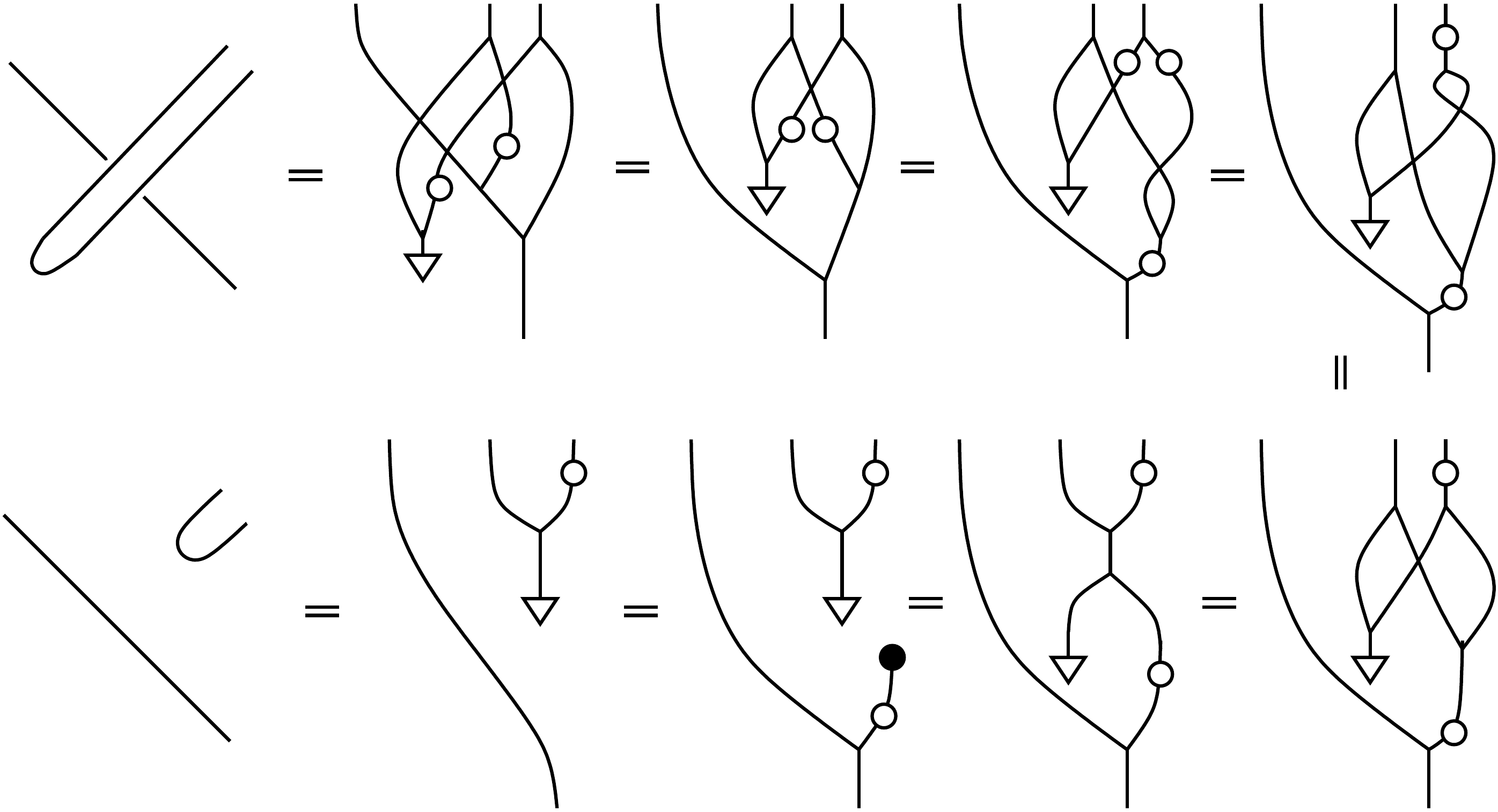}
\end{center}
\caption{Proof of $(\cup \otimes {\mathbb 1} ) \beta = {\mathbb 1} \otimes \cup$}
\label{counitb1}
\end{figure}

\begin{figure}[htb]
\begin{center}
\includegraphics[width=3in]{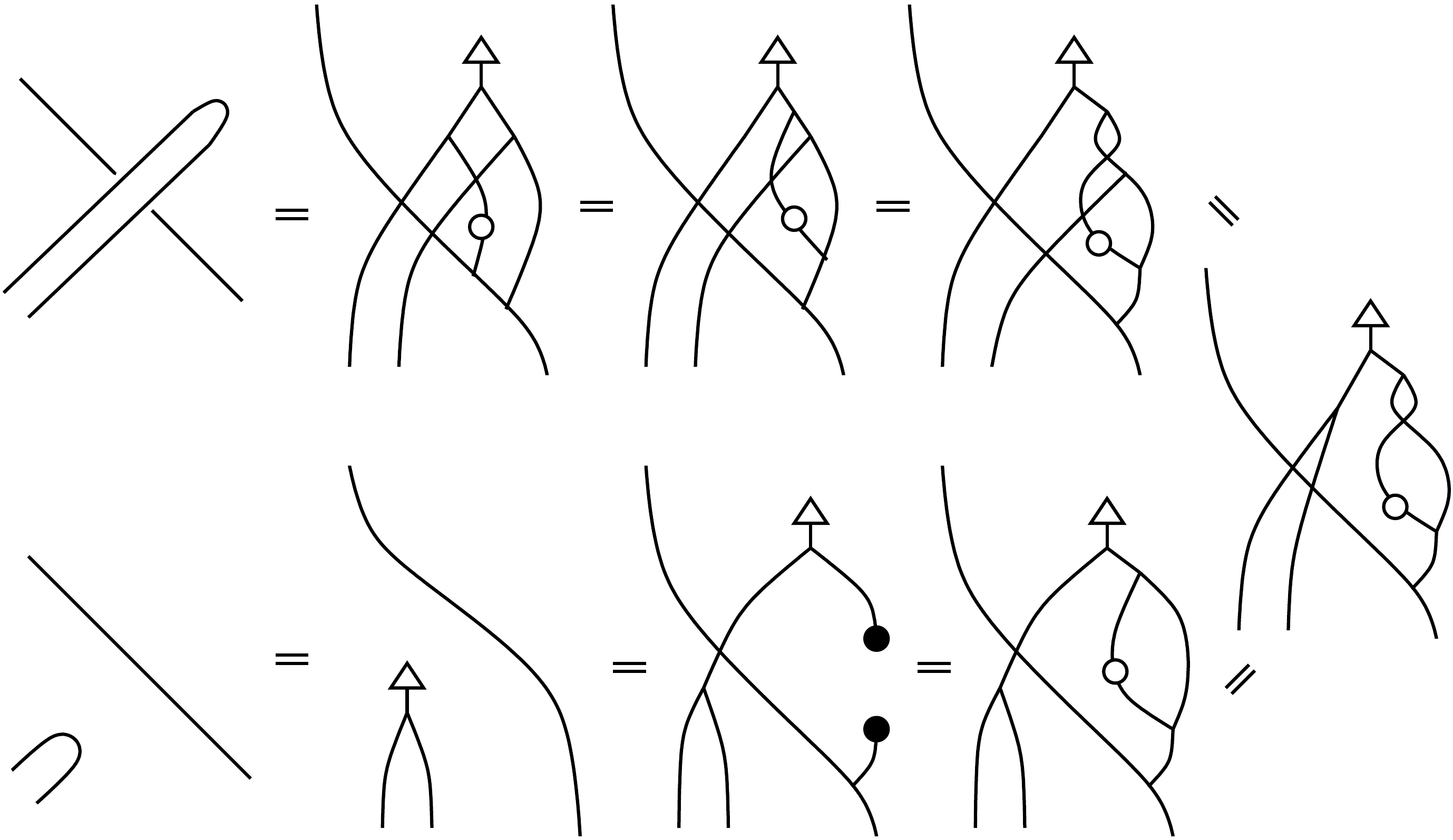}
\end{center}
\caption{Proof of $({\mathbb 1} \otimes \cap) \beta = \cap \otimes {\mathbb 1} $}
\label{unitb2}
\end{figure}

\begin{figure}[htb]
\begin{center}
\includegraphics[width=3.5in]{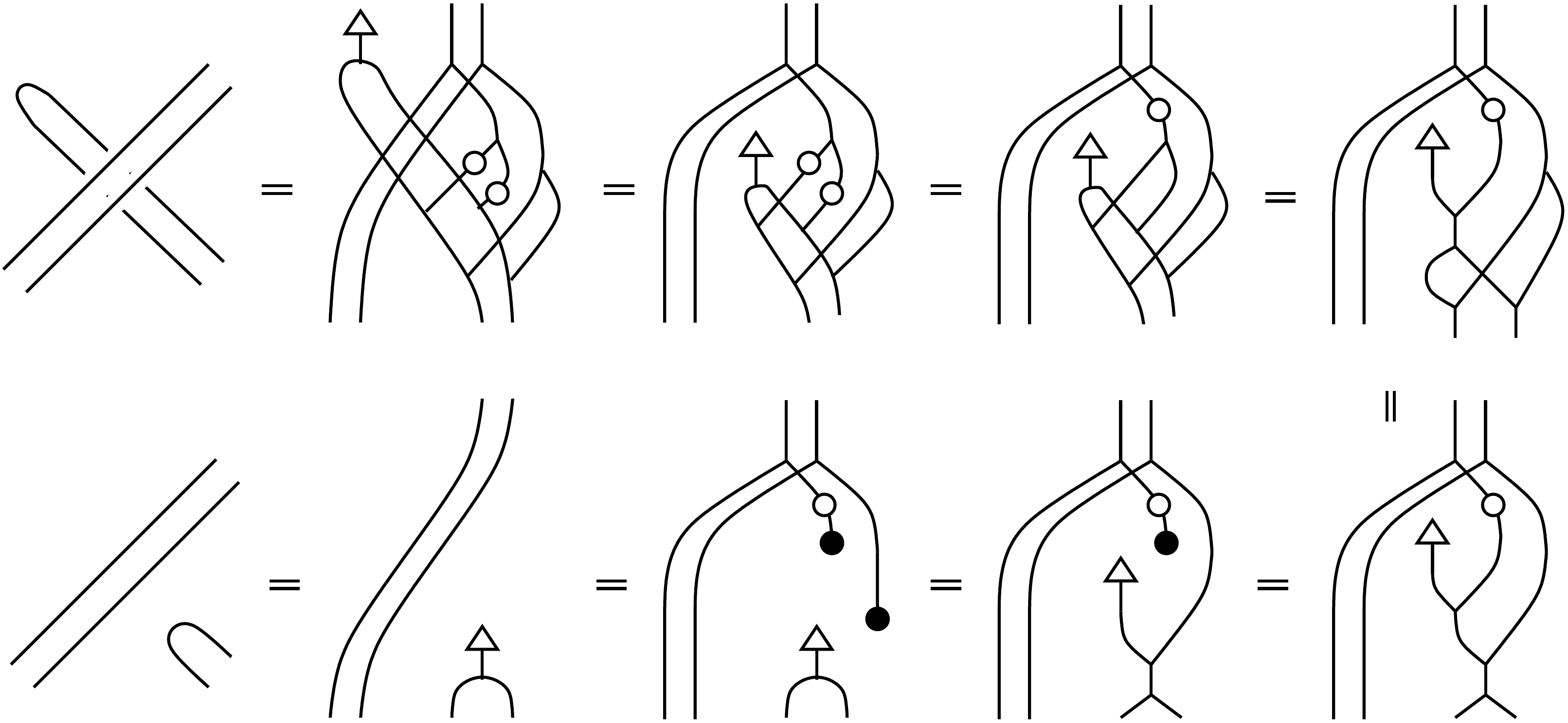}
\end{center}
\caption{Proof of $\beta (\cap \otimes {\mathbb 1} )=  {\mathbb 1} \otimes \cap$}
\label{unitb1}
\end{figure}

\section{Construction of braided Frobenius algebras}\label{sec:BraidFrob}

In \cite{Comeau}, a {\it braided Frobenius object} is defined to be a Frobenius object in a braided monoidal category. 
A monoidal category has (tensor) products among objects, with some other data and conditions,
such as unitors and associators, corresponding to units and associativity of algebras.
A monoidal category is strict if its associators and left/right unitors  are identity natural transformations.
A braided monoidal category has a braiding between (tensor) products  of  two objects, that are functorial. 
Hence it is natural to define a braided Frobenius algebra to be 
a Frobenuis object in the braided strict monoidal category of  
finitely generated modules over a unital ring.  
This definition is equivalent to having a braiding $\beta$ that commutes with all defining data of a Frobenius algebra, that corresponds to functoriality of braiding.

\begin{definition} {\rm (cf. \cite{Comeau})} 
 \label{def:braidFrob}
{\rm 
A braided Frobenius algebra is 
a Frobenuis object in the braided strict monoidal category of  
finitely generated projective modules over a unital ring.
Specifically, a braided Frobenius algebra is a Frobenius algebra $X=(V, \mu, \eta, \Delta, \epsilon)$  (multiplication, unit, comltiplication, counit)  over 
 unital ring ${\mathbb k}$, which commute with the braiding, as follows:
\begin{center}
$
\begin{array}{cc}
(\mu\otimes \mathbb 1) ({\mathbb 1} \otimes  \beta  )( \beta \otimes {\mathbb 1}) = \beta\otimes  (\mathbb 1\otimes \mu) , &
(\mathbb 1\otimes \mu) ( \beta \otimes {\mathbb 1}) ({\mathbb 1} \otimes  \beta  ) = \beta \otimes (\mu\otimes {\mathbb 1}) ,  \\
(\Delta \otimes {\mathbb 1})  \beta  =( \beta \otimes {\mathbb 1})  (\beta \otimes { \mathbb 1}) ( {\mathbb 1} \otimes \Delta )  , &
( { \mathbb 1} \otimes \Delta)  \beta  =
 (\beta \otimes { \mathbb 1}) ( \beta \otimes {\mathbb 1}) ( \Delta \otimes  {\mathbb 1}) ,\\
({\mathbb 1} \otimes \eta) \beta = \eta \otimes {\mathbb 1}, & 
\beta (\eta \otimes {\mathbb 1} )=  {\mathbb 1} \otimes \eta ,   \\\
({\mathbb 1} \otimes \epsilon) \beta = \epsilon \otimes {\mathbb 1}, & 
(\epsilon \otimes {\mathbb 1} ) \beta = {\mathbb 1} \otimes \epsilon . 
\end{array}
$
\end{center}

}
\end{definition}

\begin{figure}[htb]
\begin{center}
\includegraphics[width=5in]{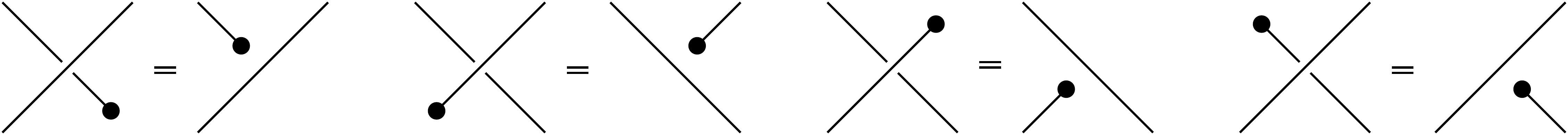}
\end{center}
\caption{Commutation of (co)unit and braiding}
\label{unitbraid}
\end{figure}

The commuting conditions for a braided Frobenius algebra for multiplication are depicted in Figure~\ref{BF}. Those for comultiplication are represented by the upside down diagrams.
The commuting conditions for the (co)unit are depicted in Figure~\ref{unitbraid}. 

We now proceed to construct a family of braided Frobenius algebras from
a class of Hopf algebras. 
We mention that the monoid structure in the next theorem also appears in Sections 4 and 5 of \cite{HV}, under the name of {\it pair of pants monoid}, for dagger pivotal categories. In our construction, the fact that Frobenius monoids (e.g. algebras) are self-dual allows us to discard the duality in $X^*\otimes X$. 

%
%

\begin{theorem}\label{thm:BraidFrob}
Let $(X,\mu,\eta,\Delta,\epsilon,S)$ be a commutative and cocommutative Hopf algebra. 
Then $V=X \otimes X$ has a braided Frobenius algebra structure.
\end{theorem}

\begin{proof}
Since $X$ is an involutory Hopf algebra, applying Lemma~\ref{lem:quantumbraid} it follows that $X\otimes X$ has a braiding $\beta$ that is induced by the quantum heap structure of $X$. We define a product $\mu_{\otimes 2}: X^{\otimes 2} \otimes X^{\otimes 2} \rightarrow X^{\otimes 2}$ by means of $\cup$ as
$$
\mu_{\otimes 2} := \mathbb 1\otimes \cup\otimes  \mathbb 1.
$$
The coproduct $\Delta_{\otimes 2}: X^{\otimes 2} \rightarrow X^{\otimes 2} \otimes X^{\otimes 2}$ is obtained from $\cap$ by the definition
$$
\Delta_{\otimes 2} := \mathbb 1\otimes \cap \otimes \mathbb 1. 
$$
Product is 
associative since 
\begin{eqnarray*}
	\mu_{\otimes 2}\circ (\mu_{\otimes 2}\otimes \mathbb 1^{\otimes 2}) &=& (\mathbb 1\otimes \cup\otimes  \mathbb 1)\circ (\mathbb 1\otimes \cup\otimes  \mathbb 1\otimes \mathbb 1^{\otimes 2})\\
	&=& \mathbb 1\otimes  \cup \otimes \cup \otimes \mathbb 1\\
	&=& (\mathbb 1\otimes \cup\otimes  \mathbb 1) \circ (\mathbb 1^{\otimes 2}\otimes \mathbb 1\otimes \cup\otimes  \mathbb 1)\\
	&=& \mu_{\otimes 2} \circ (\mathbb 1^{\otimes 2}\otimes \mu_{\otimes 2}).
	\end{eqnarray*}
Similarly, we see that $\Delta_{\otimes 2}$ is coassociative. The fact that $\mu_{\otimes 2}$ and $\Delta_{\otimes 2}$ satisfy the Frobenius laws is seen directly, as we have that $\Delta_{\otimes 2}\circ\mu_{\otimes 2}$, $(\mathbb 1^{\otimes 2}\otimes \mu_{\otimes 2})\circ (\Delta_{\otimes 2}\otimes \mathbb 1^{\otimes 2})$ and $(\mu_{\otimes 2}\otimes \mathbb 1^{\otimes 2})\circ(\mathbb 1^{\otimes 2}\Delta_{\otimes 2})$ evaluated on simple tensors $x\otimes y\otimes z\otimes w$ all equal 
$$
\cup(y\otimes z)\cdot x\otimes \cap(1)\otimes w,
$$
where we have used $\cdot$ to separate an element of the ground ring $\mathbb k$ from elements of $X^{\otimes 4}$ to avoid confusion. 

\begin{figure}[htb]
\begin{center}
\includegraphics[width=.8in]{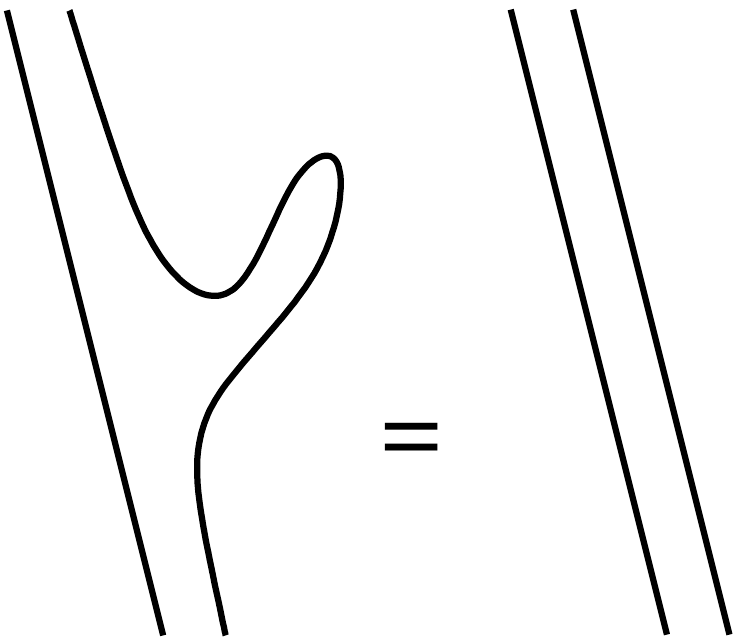}
\end{center}
\caption{The unit axiom}
\label{unit}
\end{figure}

The unit $\eta_{\otimes 2}: {\mathbb k} \rightarrow X^{\otimes 2} \otimes X^{\otimes 2}$
is defined by $\cap$. 
The unit condition follows from the switchback condition, as depicted in Figure~\ref{unit}. The counit $\epsilon_{\otimes 2}$ is defined by $\epsilon_{\otimes 2} = \cup$ and the counit condition follows similarly. 
Hence  $X\otimes X$ is endowed with a Frobenius structure and a braiding induced from the quantum heap operation.

\begin{figure}[htb]
\begin{center}
\includegraphics[width=3.5in]{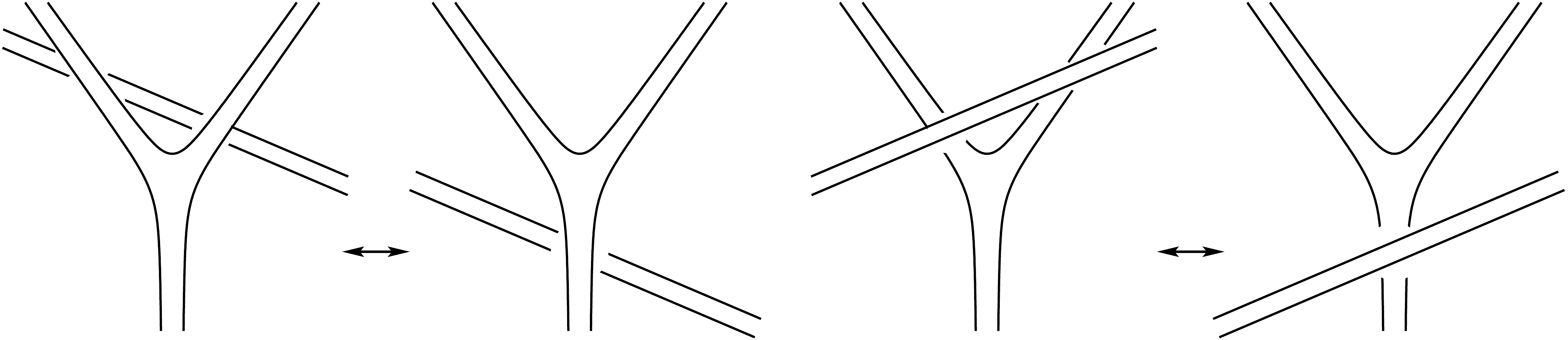}
\end{center}
\caption{Braided Frobenius conditions for a doubled Hopf algebra}
\label{BFdouble}
\end{figure}

To complete the proof we need to show that braiding and Frobenius morphisms 
commute in the sense of Definition~\ref{def:braidFrob}. 
The commutations between (co)units  and braiding follow from 
Lemma~\ref{lem:cupcapbraid} (see Figures~\ref{counitb2} through \ref{unitb1}). 
For doubled strands, the commutations between multiplication and braiding  are depicted in Figure~\ref{BFdouble}. These follow from commutations between counits  and braiding.
The commutations between comultiplication and braiding are represented by the upside down (the vertical mirror) figures of Figure~\ref{BFdouble}, and follow from 
commutations between units  and braiding.
\end{proof}

\begin{example}
{\rm
Let $X={\mathbb k}[G]$ be a group ring of a group heap $G$ 
with the TSD operation defined by linearlization of the group heap operation $T(x \otimes y \otimes z):=x y^{-1} z$ for $x,y,z \in G$. 
Endow $X$ with the Hopf algebra structure, where $\mu$ is defined by
the linearlized group multiplication,  group $\Delta(x) = x \otimes x$ for $x \in G$, 
unit defined by $\eta(1)=e \in G$ (the identity element), and counit defined by $\epsilon (x)=1$
for $x \in G$. 
The integral is defined by $\sum_{x \in G} x$ and cointegral by $e \mapsto 1$, $e \neq g \mapsto 0$. 
All conditions in Theorem~\ref{thm:BraidFrob} are checked.
The braiding is defined from $T$, and for group elements 
$\beta( (x \otimes y) \otimes (u \otimes v) ) 
= (u \otimes v)  \otimes ( x u^{-1} v \otimes y u^{-1} v )$.
Thus the braiding is the linearlization of group heap braiding as depicted in Figure~\ref{heaptypeIII}. If the group $G$ is abelian, $X$ satisfies the assumption of Theorem~\ref{thm:BraidFrob}. Moreover, so does the dual Hopf algebra  ${\mathbb k}[G]^*$.
}
\end{example}

\begin{example}
	{\rm 
Let $\mathbb k$ be a PID or a local ring of characteristic $p$. Then the truncated polynomial algebra $H = \mathbb k [X]/(X^{p^k})$ is a finitely genereated free (hence projective) Hopf algebra for any $k\geq 1$. As previously pointed out,  $H$ 
 satisfies $P(X^*)\cong {\mathbb k}$ 
since $\mathbb k$ is either a PID or a local ring. We can therefore apply Lemma~\ref{lem:switchback} and Theorem~\ref{thm:BraidFrob}, since $H$ is commutative and cocommutative. Explicitly, the algebra structure of $H$ is determined by multiplication of polynomials, the comultiplication is obtained extending $\Delta(X) = 1\otimes X + X\otimes 1$ to be an algebra homomorphism (note that it is here crucial that $H$ is truncated at a power of the characteristic of the ground ring), the counit is defined by $\epsilon (1) = 1$, $\epsilon(X) = 0$ and the antipode is given by $S(X) = -X$.  This construction can be generalized to truncated polynomial algebras with more than one indeterminate. 

We note that considering local rings gives a wider class of objects with respect to that of PID's in \cite{LS}. For instance, the ring $\Z_p[Y_1,Y_2]/(Y_1,Y_2)^2$ is a local ring that is not a PID to which the previous construction can be applied. 
}
\end{example}

\section{Twists in braided Frobenius algebras}\label{sec:Twist}

In this section we introduce  twists in  braided Frobenius algebras,  and discuss 
relations to tortile category structure and surfaces with boundary embedded in 3-space.

\begin{definition}\label{def:theta}
	{\rm 
		Let $(V, \Delta, \epsilon) $ be a finite dimensional coalgebra  over a field ${\mathbb k}$ with a TSD operation $T: V^{\otimes 3} \rightarrow V$ (Definition~\ref{def:TSD}).
		Then the operation $\theta: V \otimes V \rightarrow $ defined by 
		\begin{eqnarray*}
			\theta(x\otimes y) &=& T(x^{(1)}\otimes x^{(2)}\otimes y^{(2)})\otimes T(y^{(1)}\otimes x^{(3)}\otimes y^{(3)})
		\end{eqnarray*}
		is called a {\it twist } by $T$. 
	}
\end{definition}

\begin{remark}
	{\rm
	The twisting introduced in Definition~\ref{def:theta} is
	motivated from 
	 a ``quantum'' version of the core quandle \cite{FR} operation $(x,y) \mapsto y x^{-1} y$
	 defined on groups. In fact we have
	\begin{eqnarray*}
	x\otimes y &\mapsto& x^{(1)}S(x^{(2)})y^{(2)} \otimes y^{(1)}S(x^{(3)})y^{(3)}\\
	&=&  \epsilon (x^{(1)}) y^{(2)} \otimes y^{(1)}S(x^{(2)}) y^{(3)}\\
	&=& y^{(2)} \otimes y^{(1)}S(x)y^{(3)},
	\end{eqnarray*}
where the second term in the tensor product can be identified with the core quandle operation between $y^{(1)}$ and $x$.
}
\end{remark}

\begin{figure}[htb]
	\begin{center}
		\includegraphics[width=.5in]{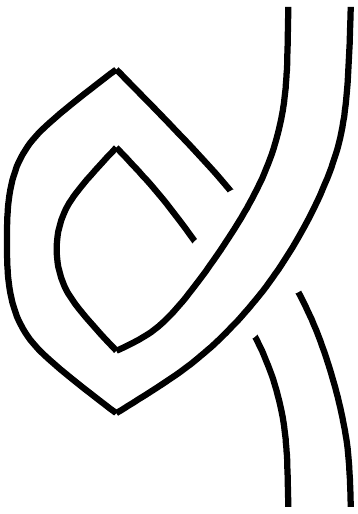}
	\end{center}
	\caption{Twisting a ribbon}
	\label{loopV}
\end{figure}

The operation $\theta$ is written by maps as follows. Fix a basis $\{ e_i : i=1, \ldots ,n\}$
for $V$, and define 
the pairing  $\vee : V \otimes V^* \rightarrow {\mathbb k} $ for the dual space $V^*$ by
$\vee (x_i \otimes x_j^*)=\delta_{i,j}$ with the Kronecker's delta, 
and copairing
$\wedge : {\mathbb k}  \rightarrow  V \otimes V^* $ by 
$\wedge (1)=\sum_{i=1}^n x_i \otimes x^*_i $. 
Then $\theta$ is written as 

$$\theta=  ({\mathbb 1}^{\otimes 2} \otimes \vee ) ( {\mathbb 1}^{\otimes 3} \otimes \vee  \otimes  {\mathbb 1}) (\beta \otimes  {\mathbb 1}^{\otimes 2} ) ( {\mathbb 1}^{\otimes 3} \otimes \wedge \otimes  {\mathbb 1}) ({\mathbb 1}^{\otimes 2} \otimes \wedge ) $$
with the braiding $\beta$ induced from $T$ 
(Lemma~\ref{lem:B}).
Diagrammatically, $\theta$ is represented by Figure~\ref{loopV}, and corresponds to 
a full twist as in the right of the figure. In the figure, the maxima and minima corresponds to $\wedge$ and
$\vee$, respectively, and indicated by such notations, to distinguish 
them from 
$\cap$ and $\cup$.

\begin{proposition} \label{prop:twist} 
	Let $(V, \Delta, \epsilon) $ be a cocommutative coalgebra  over a unital ring ${\mathbb k}$ with a TSD operation $T: V^{\otimes 3} \rightarrow V$ (Definition~\ref{def:TSD}).
	Let $\theta$ be the twist in Definition~\ref{def:theta}. 
	Then  $\theta$ commutes with the braiding $\beta$ induced from 
	$T$.
	Specifically, we have $\beta (\theta \otimes {\mathbb 1} ) = ( {\mathbb 1} \otimes \theta ) \beta$
	and $\beta  ( {\mathbb 1} \otimes \theta ) = (\theta \otimes {\mathbb 1} ) \beta$. 
\end{proposition}

\begin{proof}
On simple tensors we have 
\begin{eqnarray*}
\beta (\theta \otimes {\mathbb 1} )(x\otimes y \otimes z\otimes w) &=& \beta (T(x^{(1)}\otimes x^{(2)}\otimes y^{(2)})\otimes T(y^{(1)}\otimes x^{(3)}\otimes y^{(3)})\otimes z\otimes w )\\
&=& z^{(1)}\otimes w^{(1)} \otimes T(T(x^{(1)}\otimes x^{(2)}\otimes y^{(2)})\otimes z^{(2)}\otimes w^{(2)}) 
\\
&& \hspace{0.5cm}\otimes  T(T(y^{(1)}\otimes x^{(3)}\otimes y^{(3)})\otimes z^{(3)}\otimes w^{(3)}),
\end{eqnarray*}
and also 
\begin{eqnarray*}
( {\mathbb 1} \otimes \theta ) \beta(x\otimes y \otimes z\otimes w) &=&
({\mathbb 1} \otimes \theta) 
(  z^{(1)}\otimes w^{(1)} \otimes T(x\otimes z^{(2)}\otimes w^{(2)}) \otimes T(y\otimes z^{(3)}\otimes w^{(3)}) ) \\
&=& z^{(1)}\otimes w^{(1)} \otimes T(T(x^{(1)}\otimes z^{(21)}\otimes w^{(21)})\otimes T(x^{(2)}\otimes z^{(22)}\otimes w^{(22)}) 
\\
&& \hspace{0.5cm} \otimes T(y^{(2)}\otimes z^{(32)}\otimes w^{(32)}))\otimes T(T(y^{(1)}\otimes z^{(31)}\otimes w^{(31)}) 
\\
&& \hspace{0.5cm} \otimes T(x^{(3)}\otimes z^{(23)}\otimes w^{(23)})\otimes T(y^{(3)}\otimes z^{(33)}\otimes w^{(33)})),
\end{eqnarray*}
where the fact that, by definition, $T$ is a coalgebra morphism has been applied.
Applying cocommutativity of $\Delta$ we can rearrange the $z$ and $w$ terms in such a way that $\beta (\theta \otimes {\mathbb 1} )(x\otimes y \otimes z\otimes w) $ and $( {\mathbb 1} \otimes \theta ) \beta(x\otimes y \otimes z\otimes w)$ differ by an application of the TSD condition of $T$ utilized twice. This shows the equality $\beta (\theta \otimes {\mathbb 1} ) = ( {\mathbb 1} \otimes \theta ) \beta$. 

Let us now consider the equation $\beta  ( {\mathbb 1} \otimes \theta ) = (\theta \otimes {\mathbb 1} ) \beta$. For the LHS we have
\begin{eqnarray*}
\beta  ( {\mathbb 1} \otimes \theta ) (x\otimes y \otimes z\otimes w)  &=& T(z^{(11)}\otimes z^{(21)}\otimes w^{(21)})\otimes T(w^{(11)}\otimes z^{(31)}\otimes w^{(31)})\\
&&\hspace{0.5cm} \otimes T(x\otimes T(z^{(12)}\otimes z^{(22)}\otimes w^{(22)}) \otimes T(w^{(12)}\otimes z^{(32)}\otimes w^{(32)})\\
&&\hspace{0.5cm} \otimes T(y\otimes T(z^{(13)}\otimes z^{(23)}\otimes w^{(23)}) \otimes T(w^{(13)}\otimes z^{(33)}\otimes w^{(33)}),
\end{eqnarray*}
while for the RHS we have
\begin{eqnarray*}
(\theta \otimes {\mathbb 1} ) \beta (x\otimes y \otimes z\otimes w) &=& T(z^{(11)}\otimes z^{(12)}\otimes w^{(12)})\otimes T(w^{(11)}\otimes z^{(13)}\otimes w^{(13)})\\
&&\hspace{0.5cm} \otimes T(x\otimes z^{(2)}\otimes w^{(2)})\otimes T(y\otimes z^{(3)}\otimes w^{(3)}).
\end{eqnarray*}
To complete the proof we see that it is enough to show the equality 
\begin{eqnarray}
T(x\otimes T(z^{(1)}\otimes z^{(2)}\otimes w^{(2)})\otimes T(w^{(1)}\otimes z^{(3)}\otimes w^{(3)})) &=& T(x\otimes z\otimes w). \label{eqn:slideloop}
\end{eqnarray}
We have
\begin{eqnarray*}
\lefteqn{T(x\otimes T(z^{(1)}\otimes z^{(2)}\otimes w^{(2)})\otimes T(w^{(1)}\otimes z^{(3)}\otimes w^{(3)}))}\\
&=& \epsilon(z^{(2)})\epsilon(w^{(2)})\cdot T(x\otimes T(z^{(1)}\otimes z^{(3)}\otimes w^{(3)})\otimes T(w^{(1)}\otimes z^{(4)}\otimes w^{(4)}))\\
&=& T(T(T(x\otimes w^{(3)}\otimes z^{(3)})\otimes z^{(2)}\otimes w^{(2)})\otimes T(z^{(1)}\otimes z^{(4)}\otimes w^{(4)})\otimes T(w^{(1)}\otimes z^{(5)}\otimes w^{(5)})),
\end{eqnarray*}
where the first equality uses the definition of counit $\epsilon$, and the second equality makes use of the invertibility condition of $T$. Let us now apply the TSD property of $T$ to the terms $T(x\otimes w^{(2)}\otimes z^{(2)})$, $z^{(1)}$, $w^{(1)}$, $z^{(3)}$ and $w^{(3)}$, where we set 
$T(x\otimes w^{(3)}\otimes z^{(3)}) = q$ 
for convenience. We get 
\begin{eqnarray*}
	\lefteqn{T(T(q\otimes z^{(2)}\otimes w^{(2)})\otimes T(z^{(1)}\otimes z^{(4)}\otimes w^{(4)})\otimes T(w^{(1)}\otimes z^{(5)}\otimes w^{(5)}))  } \\
	&&=T(T(q\otimes z^{(2)}\otimes w^{(2)})\otimes z^{(1)}\otimes w^{(1)})
	 =\epsilon(z^{(1)})\epsilon(w^{(1)})\cdot T(x\otimes z^{(2)}\otimes w^{(2)})
	= T(x\otimes z\otimes w),
	\end{eqnarray*}
where invertibility of $T$, as well as cocommutativity of $\Delta$, has been used in the second equality. 
This shows that Equation~(\ref{eqn:slideloop}) holds. 
\end{proof}

\begin{figure}[htb]
	\begin{center}
		\includegraphics[width=1.5in]{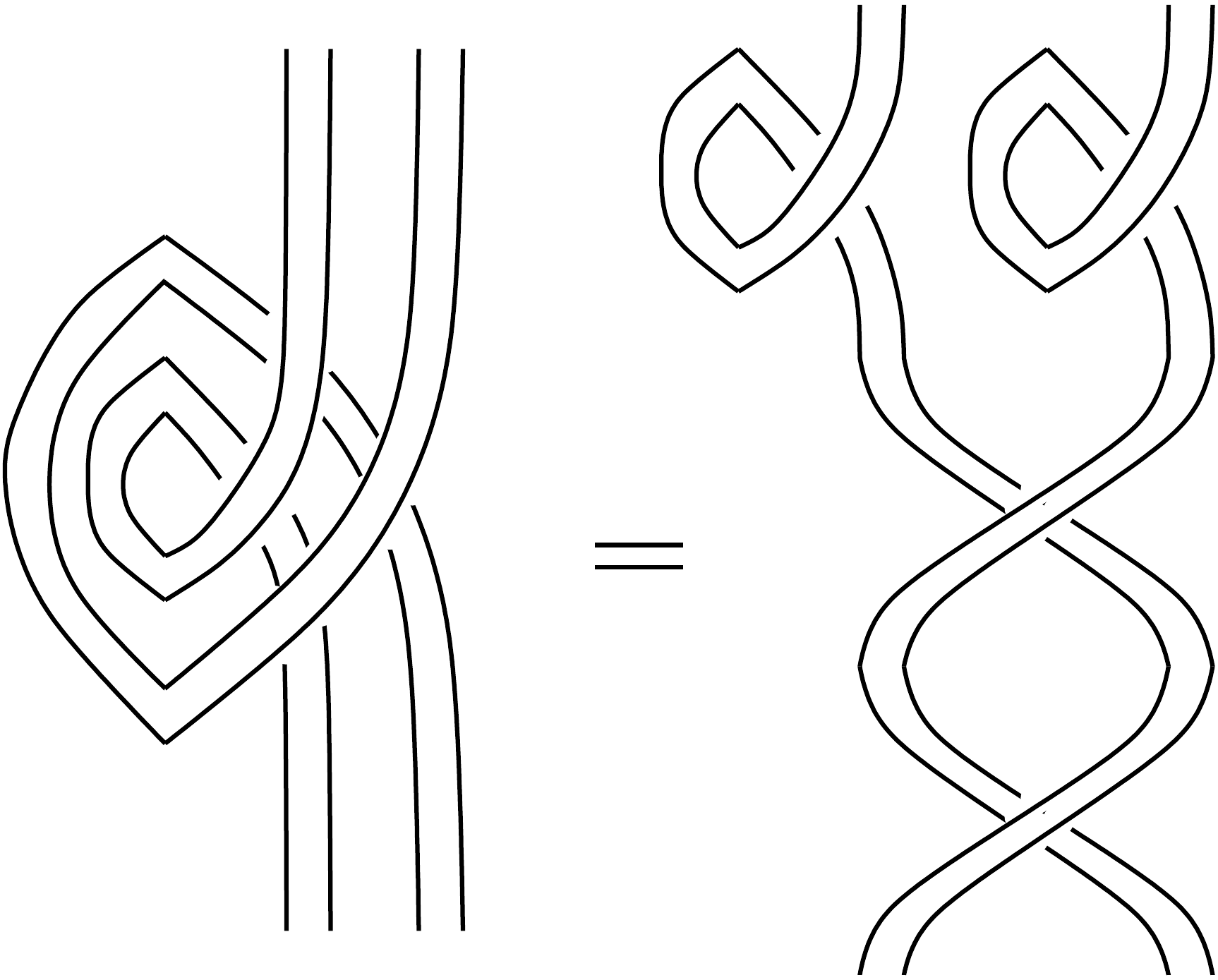}
	\end{center}
	\caption{Twisting a doubled ribbon}
	\label{loopdouble}
\end{figure}

\begin{remark} \label{rem:tortile}
	{\rm
		Here we discuss relations to the {\it tortile category}.
		A braided monoidal category is called {\it tortile} \cite{JS} (or {\it ribbon} \cite{HV}) 
		if there is a morphism $\theta_X$ called a {\it twist}  for every object $X$
		such that $\theta_{X, Y} = \beta_{Y, X} \beta_{X, Y} ( \theta_X \otimes \theta_Y) $ for all objects $X, Y$, where $\beta$ denotes the braiding.
		
		Let $(V, \Delta, \epsilon) $ be a finite dimensional coalgebra  over a field ${\mathbb k}$ with a TSD operation $T: V^{\otimes 3} \rightarrow V$ (Definition~\ref{def:TSD}).
		Then Proposition~\ref{prop:twist} implies that the subcategory generated by $V$ in the category of 
		braided monoidal category of finite dimensional coalgebras with TSDs forms a tortile category.
		The twist  $\theta_{\otimes k}$ on $V^{\otimes k}$ is defined by parallel loops, that are defined by taking $k$-fold parallel ribbons. The case $k=2$ is depicted in Figure~\ref{loopdouble} left.
		The equality 
		$\theta_{V, V} = \beta_{V, V} \beta_{V, V} ( \theta_V \otimes \theta_V)$
		is indicated in the figure.
		The fact that full twist of parallel strings form a tortile category is pointed out in \cite{JS}.
		In \cite{JS}  the twists are defined by parallel loops, that topologically  correspond to full twists 
		of parallel strings, using dual spaces. Thus the construction of this twists are obtained by
		applying the twists in \cite{JS} to braiding defined by TSD operations on coalgebras.
	}
\end{remark}

\begin{figure}[htb]
	\begin{center}
		\includegraphics[width=2.2in]{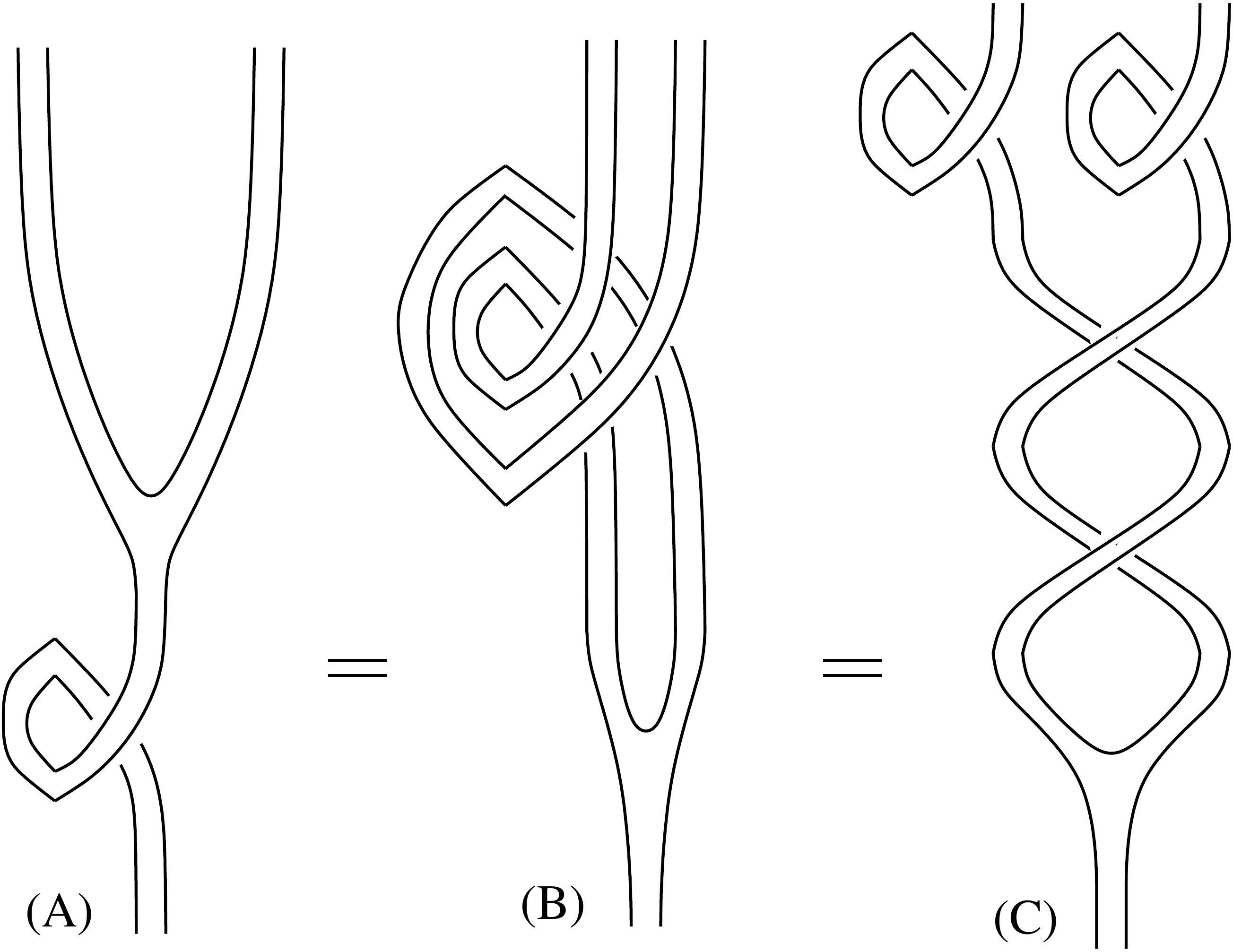}
	\end{center}
	\caption{Commutation between a twist and multiplication}
	\label{loopmultiV}
\end{figure}

\begin{proposition}	
Let $X$ be as in Theorem~\ref{thm:BraidFrob}, and let $V = X\otimes X$ denote the associated braided Frobenius structure on the doubled vector space. 
	Let $\theta$ be the twist in Definition~\ref{def:theta}. 
Then the   twist  $\theta$ commutes with the multiplication and comultiplication. This means,  with notations as in Remark~\ref{rem:tortile}, that $\theta_V \, \mu_{\otimes 2} =  \mu_{\otimes 2} \, \theta_{  V , V }$
and  $ \Delta_{\otimes 2} \theta_V = \theta_{V,V} \Delta_{\otimes 2}$ hold.
\end{proposition}

The commutation between the twist and multiplication is depicted in the left equality $ (A) = (B)$  of Figure~\ref{loopmultiV}. The right equality $(B)=(C)$  is a consequence of Figure~\ref{loopdouble}.
We note that the resulting equality $ (A)=(C) $ corresponds 
diagrammatically to twisting the trivalent vertex by one full twist.

\begin{proof}
We verify equality $\theta_V \, \mu_{\otimes 2}  =  \mu_{\otimes 2}  \, \theta_{  V , V }$ on simple tensors $x\otimes y \otimes z\otimes w$. For the LHS we have
\begin{eqnarray*}
\theta_V \, \mu_{\otimes 2}  (x\otimes y\otimes z\otimes w) &=& \gamma (yS(z)) \cdot w^{(2)}\otimes w^{(1)}S(x)w^{(3)}.
\end{eqnarray*}
The RHS is given as
\begin{eqnarray*}
\mu_{\otimes 2}  \, \theta_{  V , V }(x\otimes y\otimes z\otimes w) &=& \gamma(y^{(1)}S(x^{(3)}) y^{(3)}S(z^{(3)})w^{(3)}S(w^{(4)})z^{(4)}S(y^{(4)})x^{(4)}S(z^{(1)}))\\
&& \cdot x^{(1)}S(x^{(2)})y^{(2)}S(z^{(2)})w^{(2)}\otimes w^{(1)}S(x^{(5)})y^{(5)}S(z^{(5)})w^{(5)}\\
&=&  \gamma (y^{(1)}S(z^{(1)}))\cdot y^{(2)}S(z^{(2)})w^{(2)}\otimes w^{(1)}S(x)y^{(3)}S(z^{(3)})w^{(3)}\\
&=& \gamma (yS(z))\cdot w^{(2)}\otimes w^{(1)}S(x)w^{(3)},
\end{eqnarray*}
where the first equality is obtained by unraveling the definitions, the second equality is a multiple application of the counit axiom, and the third equality follows by applying the definition of cointegral $\gamma$ twice. Equality $ \Delta_{\otimes 2}  \theta_V = \theta_{V,V} \Delta_{\otimes 2} $ is proven on 
simple 
 tensors in a similar fashion.
%
\end{proof}

\begin{figure}[htb]
	\begin{center}
		\includegraphics[width=1.2in]{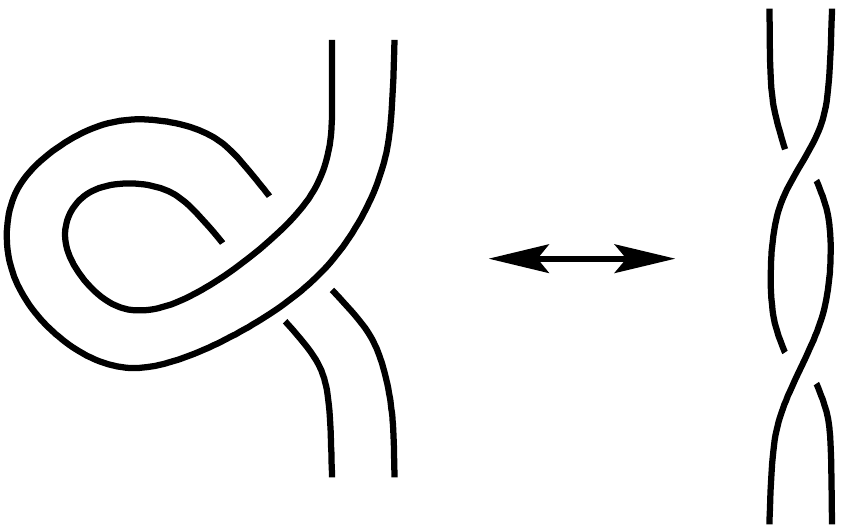}
	\end{center}
	\caption{Twisting a ribbon by a loop}
	\label{loop}
\end{figure}

\begin{remark}
	{\rm 
		
		For the braided Frobenius algebra $V$ constructed in Theorem~\ref{thm:BraidFrob}, a twist $\Theta$ can be defined using $\cap$ and $\cup$ instead of $\wedge$ and $\vee$ as depicted in Figure~\ref{loop} left. 
		Specifically, 
		$$\Theta=  ({\mathbb 1}^{\otimes 2} \otimes \cup ) ( {\mathbb 1}^{\otimes 3} \otimes \cup  \otimes  {\mathbb 1}) (\beta \otimes  {\mathbb 1}^{\otimes 2} ) ( {\mathbb 1}^{\otimes 3} \otimes \cap  \otimes  {\mathbb 1}) ({\mathbb 1}^{\otimes 2} \otimes \cap ) $$ with the braiding $\beta$ induced from $T$ 
		(Lemma~\ref{lem:B}).
		Since all maps that appear in this formula commute with the braiding $\beta$ from earlier lemmas,
		$\Theta$ commute with $\beta$. By the same argument as Remark~\ref{rem:tortile}, 		we obtain a tortile category from $V$.
		Similarly, $\Theta$ commutes with $\mu$ and $\Delta$. A sketch proof of 
		the commutation between $\mu$ and $\Theta$ is depicted in Figure~\ref{loopmultiproof}.
	}
\end{remark}

\begin{figure}[htb]
	\begin{center}
		\includegraphics[width=3.5in]{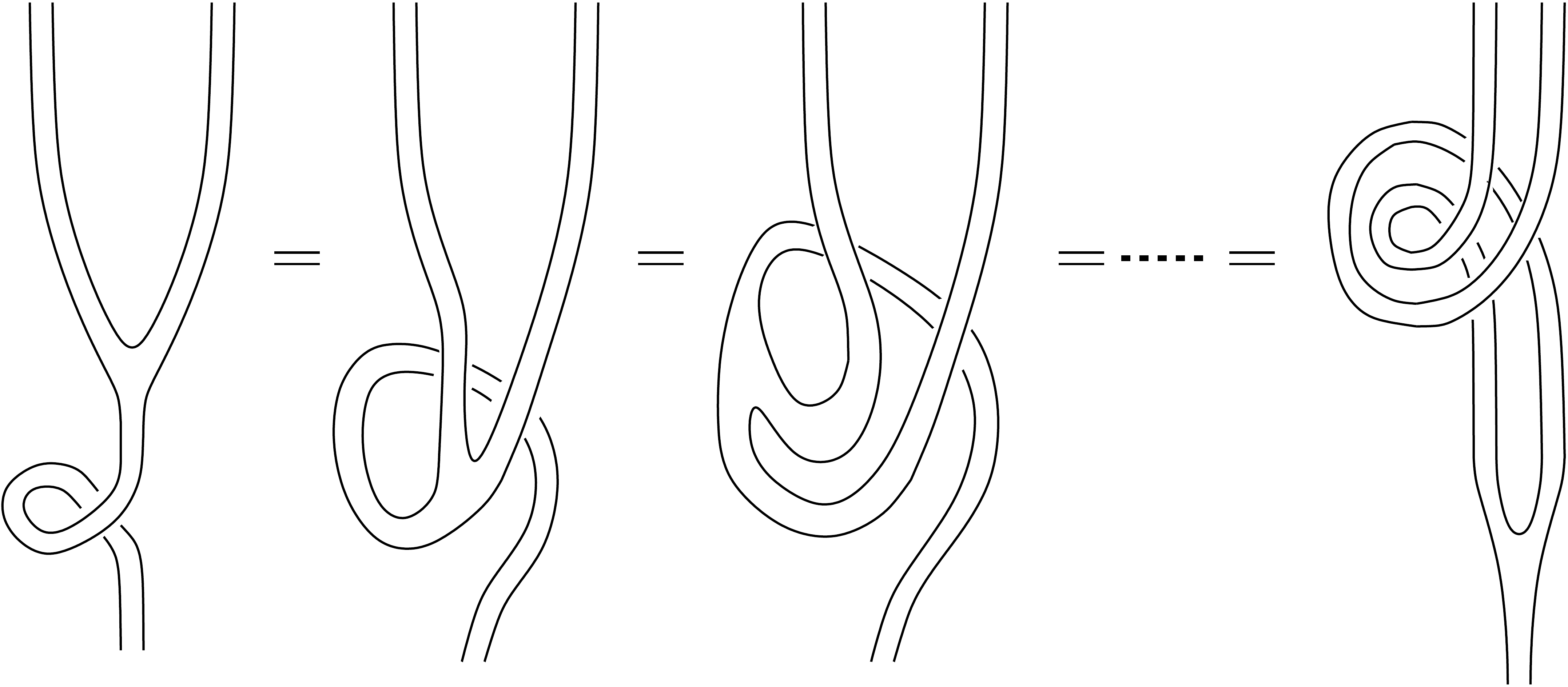}
	\end{center}
	\caption{Sketch picture proof of commutation}
	\label{loopmultiproof}
\end{figure}


We close the paper with remarks on invariants of embedded surfaces with boundary.
It is of interest to find invariants of compact orientable surfaces with boundary represented by ribbon graph diagrams,
as considered in \cite{Matsu}, in a way analogous to quantum invariants using braided Frobenius algebras.
In this approach, a height function is fixed on the plane, and building blocks of diagrams consist of cups and caps in addition to crossings and trivalent vertices.
Although a complete  set of moves for  ribbon graph diagrams for certain embedded surfaces 
was given in \cite{Matsu},
  height functions were not considered.
 It is desirable to have a list of additional moves. For example, the passcup move and passcap move (the upside down of passcup) are such moves, and they are satisfied by braided Frobenius algebras constructed in this paper. Another move depicted in Figure~\ref{capmult} is also satisfied from Frobenius algebra axioms. Although most moves in \cite{Matsu} for orientable surfaces (without half twists), with appropriate choices of height functions, are satisfied by our resulting braided Frobenius algebras, it is not clear at this time whether the equation corresponding to the move depicted in Figure~\ref{cancelpair} is satisfied, in general, under our construction. 
However, it may be satisfied by some specific examples,
and may provide invariants for such surfaces.

\begin{figure}[htb]
	\begin{center}
		\includegraphics[width=1.8in]{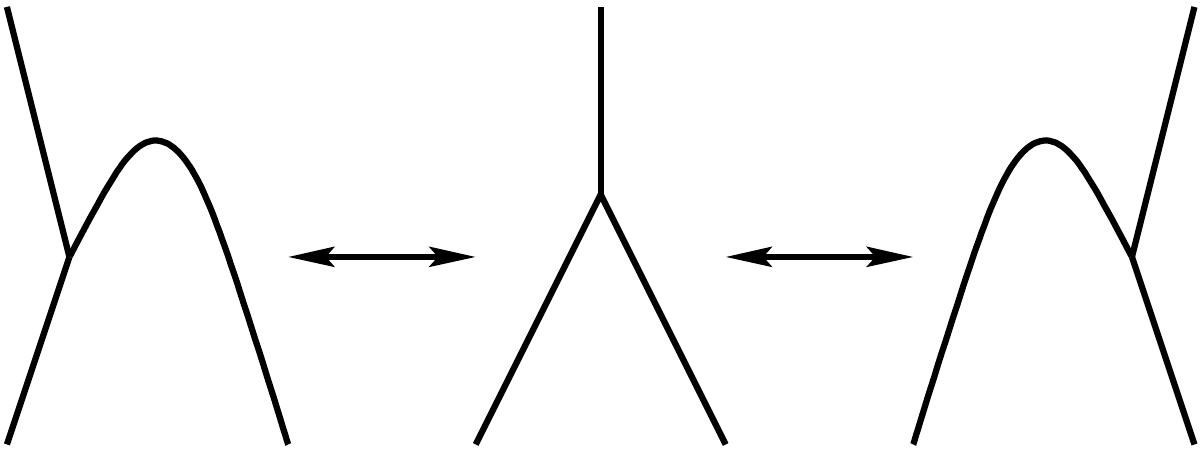}
	\end{center}
	\caption{Conversion of $\mu$ to $\Delta$ through $\cap$}
	\label{capmult}
\end{figure}

\begin{figure}[htb]
	\begin{center}
		\includegraphics[width=1.2in]{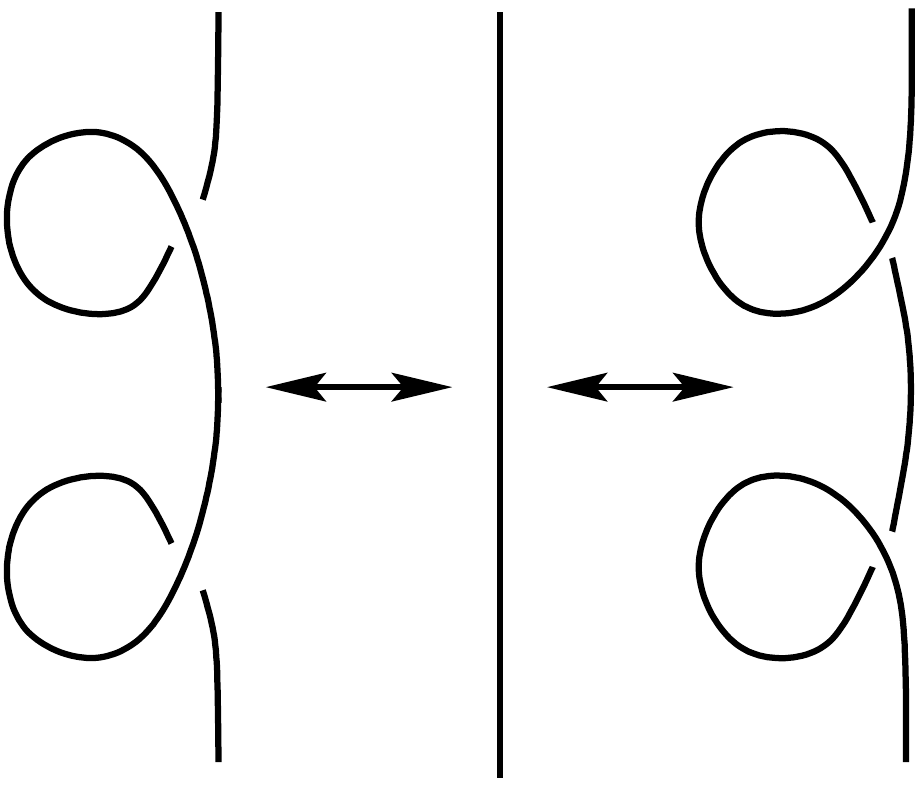}
	\end{center}
	\caption{Canceling a pair of loops}
	\label{cancelpair}
\end{figure}

For non-orientable surfaces, 
ribbon graph diagrams~\cite{Matsu} contain 
half-twists, and there is a  move of twisting a vertex  as indicated in Figure~\ref{twistvertdble},
that involve half twists of ribbons merging at a vertex.
From the topological correspondence of the twist $\theta$ to a full twists as in Figure~\ref{loop},
such a hypothetical half twist, which we denote by $\sqrt{\theta}$, would be required to 
satisfy $\sqrt{\theta} \circ \sqrt{\theta} =\theta$ (thus the notation). 
We have not found such a morphism in braided Frobenius algebras constructed in Theorem~\ref{thm:BraidFrob}, and raise a question:
For the twists ($\theta$ and $\Theta$) defined in this section for 
the  braided Frobenius algebras constructed in Theorem~\ref{thm:BraidFrob},
are there  half twists $\sqrt{\theta}$ and $\sqrt{\Theta}$ ?
We point out  a curious fact that the composition of a half-twist of a vertex in Figure~\ref{twistvertdble}
twice is a full twist of a vertex represented by Figure~\ref{loopmultiV}, which is satisfied by 
the braided Frobenius algebras constructed in this paper.

\begin{figure}[htb]
	\begin{center}
		\includegraphics[width=2.2in]{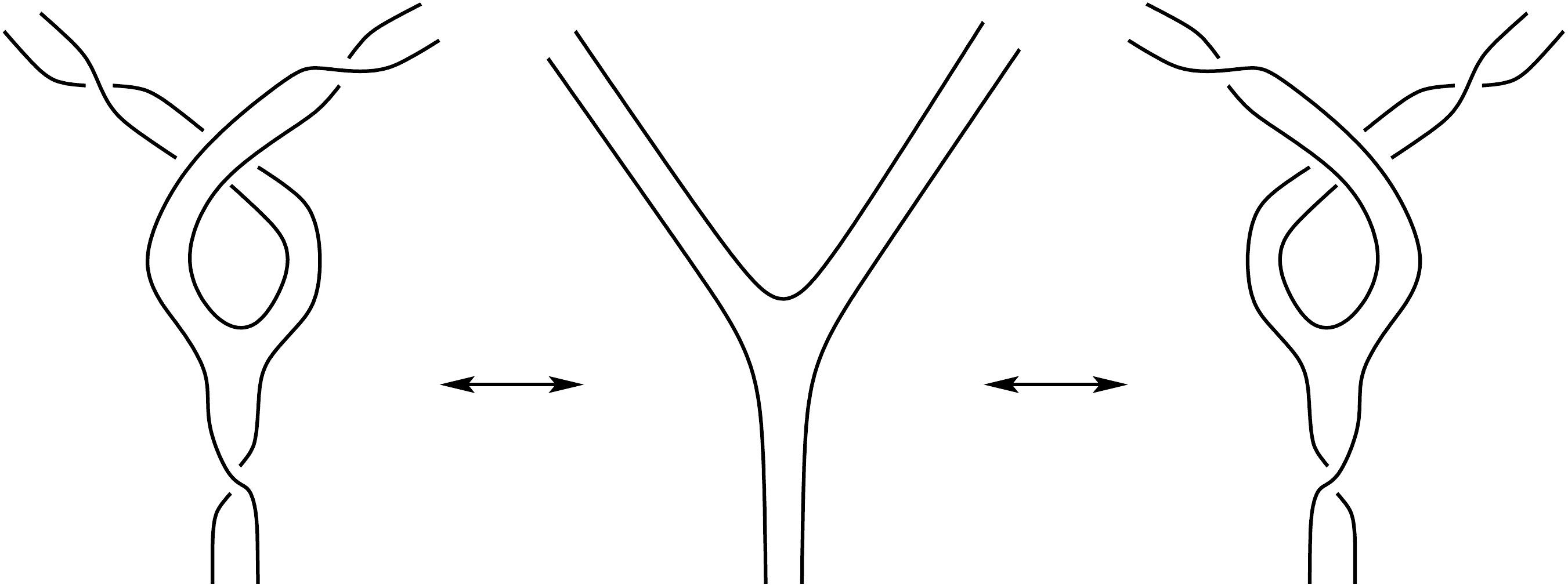}
	\end{center}
	\caption{Twisting a vertex of a ribbon}
	\label{twistvertdble}
\end{figure}

\smallskip

\noindent
{\bf Acknowledgements.} We are thankful to J. Scott Carter and Atsushi Ishii for valuable information.
MS was supported in part by NSF DMS-1800443.
EZ was 
supported by the Estonian Research Council via the Mobilitas Pluss scheme, grant MOBJD679.

\end{document}